\documentclass[12pt,reqno]{amsart}
\usepackage{}
\usepackage{amsfonts}
\usepackage{a4wide}
\allowdisplaybreaks
\numberwithin{equation}{section}

\newtheorem{theorem}{Theorem}[section]
\newtheorem{proposition}[theorem]{Proposition}

\newtheorem{lemma}[theorem]{Lemma}

\theoremstyle{definition}

\newtheorem{remark}[theorem]{Remark}

\newcommand{\R}{\mathbb{R}}
\newcommand{\dis}{\displaystyle}

\begin{document}

\title
 [Kirchhoff type problems with critical Sobolev exponents]
 {Concentrating Bound States for Kirchhoff type problems in ${\R^3}$ involving critical Sobolev exponents}

 \author{Yi He,\,\,\,\,Gongbao Li\,\,and\,\,Shuangjie Peng}

 \address{School of Mathematics and Statistics, Central China
Normal University, Wuhan, 430079, P. R. China }

\email{ heyi19870113@163.om}

\address{School of Mathematics and Statistics, Central China
Normal University, Wuhan, 430079, P. R. China }

\email{ ligb@mail.ccnu.edu.cn}

\address{School of Mathematics and Statistics, Central China
Normal University, Wuhan, 430079, P. R. China }

\email{ sjpeng@mail.ccnu.edu.cn}

\begin{abstract}
We study the concentration and multiplicity of weak solutions to the
Kirchhoff type equation with critical Sobolev growth
\[\left\{ \begin{gathered}
   - \Bigl({\varepsilon ^2}a + \varepsilon b\int_{{\R^3}} {{{\left| {\nabla u} \right|}^2}} \Bigr)\Delta u + V(z)u
   = f(u) + {u^5}{\text{ in }}{\R^3}, \hfill \\
  u \in {H^1}({\R^3}),{\text{ }}u > 0{\text{ in }}{\R^3}, \hfill \\
\end{gathered}  \right.\]
where $\varepsilon $ is a small positive parameter and $a,b > 0$ are
constants, $f \in {C^1}({\R^ + },\R)$  is subcritical, $V:{\R^3} \to
\R$ is a locally H\"{o}lder continuous function.
 We first prove that for ${\varepsilon _0} > 0$  sufficiently small, the above
   problem has a weak solution ${u_\varepsilon }$ with
  exponential decay at infinity. Moreover, ${u_\varepsilon }$ concentrates around
  a local minimum point of  $V$ in $\Lambda $  as $\varepsilon  \to 0$. With minimax theorems and Ljusternik-Schnirelmann theory, we also obtain multiple solutions by employing the
   topology construct of the set where
  the potential $V\left( z \right)$ attains its minimum.


\end{abstract}

\maketitle

\section{Introduction and Main Result}
In this paper, we study the Kirchhoff type equation
\[
\left\{ \begin{gathered}
   - \Bigl(({\varepsilon ^2}a + \varepsilon b\int_{{\R^3}} {|\nabla u{|^2}} \Bigr)\Delta u + V(z)u = f(u) + {u^5}{\text{ in }}{\R^3}, \hfill \\
  u \in {H^1}({\R^3}),{\text{ }}u > 0{\text{ in }}{\R^3}, \hfill \\
\end{gathered}  \right.\tag{${E_\varepsilon }$}
\]
where $\varepsilon $ is a small positive parameter, $a,b > 0$ are
constants and $f$ is a continuous subcritical and superlinear
nonlinearity. Such problems are often referred to as being nonlocal
because of the presence of the term $(\int_{\R^3}|\nabla u|^2 )\Delta
u$ which implies that problem ($E_\varepsilon$) is no longer a
pointwise identity. Problem ($E_\varepsilon$) is a variant type of
 the following Dirichlet problem of Kirchhoff type
\begin{equation}\label{1.5}
\left\{ \begin{gathered}
   - \Bigl( {a + b\int_\Omega  {{{| {\nabla u} |}^2}} } \Bigr)\Delta u
    = f({z,u}){\text{ in }}\Omega , \hfill \\
  u = 0{\text{ on }}\partial \Omega , \hfill \\
\end{gathered}  \right.
\end{equation}
which is related to the stationary analogue of the equation
\begin{equation}\label{1.6}
\left\{ \begin{gathered}
  {u_{tt}} - \Bigl( {a + b\int_\Omega  {{{\left| {\nabla u} \right|}^2}} } \Bigr)\Delta u
  = f({z,u}){\text{ in }}\Omega , \hfill \\
  u = 0{\text{ on }}\partial \Omega , \hfill \\
\end{gathered}  \right.
\end{equation}
proposed by Kirchhoff in \cite{k} as an existence of the classical
D'Alembert's wave equations for free vibration of elastic strings.
Kirchhoff's model takes into account the changes in length of the
string produced by transverse vibrations. In ($E_\varepsilon$),  $u$
denotes the displacement, $f({z,u})$ the external force and $b$ the
initial tension while $a$ is related to the intrinsic properties of
the string (such as Young's modulus). We have to point out that
nonlocal problems also appear in other fields as biological systems,
where $u$ describes a process which depends on the average of itself
(for example, population density). After the pioneer work of Lions
\cite{l4}, where a functional analysis approach was proposed,
 Problem ($E_\varepsilon$) began to call
attention of several researchers. In \cite{ap}, Arosio and Panizzi
studied the Cauchy-Dirichlet type problem related to \eqref{1.6} in
the Hadamard sense as a special case of an abstract second-order
Cauchy problem in a Hilbert space. Ma and Rivera In \cite{mr}
obtained positive solutions of such problems by using variational
methods.  A nontrivial solution of \eqref{1.5} was obtained via Yang index and
critical group by Perera and Zhang in \cite{pz}. In \cite{hz1}, He
and Zou obtained infinitely many solutions of \eqref{1.5} by using
local minimum method and the fountain theorem. In \cite {ckw}, \eqref{1.5} was studied
with concave and convex nonlinearities by using Nehari manifold and
fibering map methods, and multiple
positive solutions were obtained. For more result, we can refer to \cite{ap,ds,mr}
and the references therein.

We note that problem $({E_\varepsilon })$ with  $b = 0$ is motivated
by the search for standing wave solutions  for the nonlinear
Schr\"{o}dinger equation, which  is one of the main subjects in
nonlinear analysis. Different approaches have been taken to deal
with this problem under various hypotheses on the potentials and the
nonlinearity (see \cite{fw,o1,o2,o3,r,dn,w,pf} and so on).

For $({{E_\varepsilon }})$ without the critical growth, it seems
that the first existence result of  concentration solutions and
multiple solutions  for small $\varepsilon$ was obtained by He and
Zou  in \cite{hz2}. While for the critical growth, in \cite{wtxz},
Wang, Tian, Xu and Zhang considered ($E_\varepsilon$) with $f(u)$
replaced by $\lambda f(u)$ and obtained some interesting results, where
$\lambda>0$ is a large parameter. It was proved in \cite{wtxz} that
ground state solutions and multiple solutions exist for large
$\lambda>0$ under the condition $\inf_{x\in\R^3}V(x)<\liminf_{|x|\to
\infty} V(x)$. Moreover, if $\inf_{x\in\R^3}V(x)=\liminf_{|x|\to
\infty} V(x)=V^\infty$  and $V(v)\not\equiv V^\infty$,
($E_\varepsilon$) does not has ground state solutions. We point out
here that to overcome the obstacle due to the appearance of the
critical nonlinearity $u^5$,  the parameter $\lambda>0$ should be
large enough in \cite{wtxz}.

In this paper, we will consider ($E_\varepsilon$) (without the
parameter $\lambda$ before $f(u)$) and study the existence of
concentration solutions  in the case that $V(x)$ has local minimum
points. Our assumptions are as follows.

 $V$ is a locally H\"{o}lder continuous
function satisfying for some positive constant $\alpha $,
\[
V(z) \ge \alpha  > 0{\text{ for all }}z \in {\R^3}\tag{$V_1$}
\]
and
\[
\mathop {\inf }\limits_\Lambda  V < \mathop {\min }\limits_{\partial \Lambda } V\tag{$V_2$}
\]
for some open bounded set $\Lambda $.

$f \in {C^1}({{\R^ + },\R})$ satisfies: \\
$(f_1)$ $f(s) = o({s^3})$ as $s \to {0^ + }$;\\
$(f_2)$ The function $\frac{{f(s)}}
{{{s^3}}}$ is strictly increasing for $s > 0$;\\
$(f_3)$ $\exists \lambda  > 0$ such that
$f(s) \ge \lambda {s^{{q_1}}}$ for some $3 \le {q_1} < 5$ (If ${q_1} = 3$, we require a sufficiently large $\lambda $, otherwise $\lambda $ can be fixed);\\
$(f_4)$ $f(s) \le C(1 + {| s |^{q - 1}})$ for some $C>0$ where $4 < q < 6$.

It follows from $(f_1)$, $(f_2)$ that
\begin{equation}\label{1.1}
0 < 4F(s) \le f(s)s{\text{ for all }}s > 0,
\end{equation}
where $F(s) = \int_0^s {f(\tau)} d\tau $.

As we are interested in positive solutions, we define $f(s) = 0$ for $s  \le 0$.

We define
\[
H: = \Bigl\{ {u \in {H^1}({\R^3})| {\int_{{\R^3}} {V(z){u^2}}  <
\infty } .} \Bigr\},
\]
 with the norm
\[
{\| u \|_H} = {\Bigl( {\int_{{\R^3}} {a{{| {\nabla u} |}^2} +
V(z){u^2}}}\Bigr)^{\frac{1} {2}}}.
\]

We call $u \in H$ a weak solution
to $({E_\varepsilon })$ if for any $\varphi  \in H$ it holds that
\[
\bigl( {{\varepsilon ^2}a + \varepsilon b\int_{{\R^3}} {{{| {\nabla
u} |}^2}} } \bigr)\int_{{\R^3}} {\nabla u\nabla \varphi  +
\int_{{\R^3}} {V(z)u\varphi } }  = \int_{{\R^3}} {\bigl( {f(u) +
{{({{u^ + }})}^5}} \bigr)\varphi } .
\]
For $I \in {C^1}({H,{\R}})$, we say that $I$ satisfies Palais-Smale
condition ($(P.S.)$ condition in short) if any sequence $\{{{u_n}}\}
\subset H$ with $I({{u_n}})$ bounded, $I'({{u_n}}) \to 0$, has a
convergent subsequence in $H$.

 Our main results are as follows:

\begin{theorem} \label{1.1.}
Suppose that the potential $V$ satisfies $(V_1)$, $(V_2)$ and $f \in {C^1}({{\R^ + },\R})$
satisfies $(f_1)$-$(f_4)$. Then there is an ${\varepsilon _0} > 0$
such that problem $({{E_\varepsilon }})$ possesses a positive
weak solution ${u_\varepsilon } \in H$ for
all $\varepsilon  \in ({0,{\varepsilon _0}}]$. Moreover, ${u_\varepsilon }$ possesses
a maximum ${z_\varepsilon } \in \Lambda $ such
that $V({{z_\varepsilon }}) \to \mathop {\inf }\limits_\Lambda  V$, as $\varepsilon  \to 0$, and
\begin{equation}\label{1.7}
{u_\varepsilon } \le \alpha \exp \Bigl( { - \frac{\beta }
{\varepsilon }| {z - {z_\varepsilon }} |} \Bigr),{\text{ }}z \in
{\R^3}{\text{ and }} \varepsilon  \in ( {0,{\varepsilon _0}}],
\end{equation}
for some positive constants $\alpha $, $\beta $.
\end{theorem}

In order to get the multiple solutions
for $({{E_\varepsilon }})$, we need one more assumption:
\[
M: = \Bigl\{ {z \in \Lambda |{\text{ }}V(z) = \mathop {\inf }
\limits_{\xi  \in {\R^3}} V(\xi )} \Bigr\} \ne \emptyset
.\tag{$V_3$}
\]
We recall that, if $Y$ is a closed set of a
topological space of $X$, ${\text{ca}}{{\text{t}}_X}(Y)$
is the Ljusternik-Schnirelmann category of $Y$ in $X$,
namely the least number of closed and contractible
sets in $X$ which cover $Y$. We denote by
\[
{M_\delta }: = \Bigl\{ {z \in {\R^3}| {{\text{dist}}({z,M}) \le
\delta }}\Bigr\}
\]
the closed $\delta $-neighborhood of $M$, and we shall prove the
following multiplicity result

\begin{theorem} \label{1.2.}
Suppose that the potential $V$ satisfies $(V_1)$, $(V_2)$, $(V_3)$,
and $f \in {C^1}({{\R^ + },\R})$ satisfies $(f_1)$-$(f_4)$.
Then, for any $\delta  > 0$ given, there exists
${\varepsilon _\delta } > 0$ such that, for any
$\varepsilon  \in ({0,{\varepsilon _\delta }})$, the
Equation $({{E_\varepsilon }})$ has at
least ${\text{ca}}{{\text{t}}_{{M_\delta }}}(M)$
solutions. Furthermore, if ${u_\varepsilon }$ denotes
one of these solutions and ${u_\varepsilon }$ possesses
a maximum ${z_\varepsilon } \in \Lambda $, then\\
(i) $\mathop {\lim }\limits_{\varepsilon  \to 0 } V({{z_\varepsilon }})
= \mathop {\inf }\limits_\Lambda  V$;\\
(ii) ${u_\varepsilon } \le \alpha \exp({ - \frac{\beta }
{\varepsilon }| {z - {z_\varepsilon }} |})$ for all $z \in {\R^3}$
and $\varepsilon  \in ({0,{\varepsilon _\delta }})$ for some
positive constants $\alpha $, $\beta $.
\end{theorem}

\begin{remark} \label{1.3.}
If we replace $\R^3$ by $\Omega $, where $\Omega$ is a smooth domain
in $\R^3$ (possibly unbounded), then  Theorem \ref{1.1.} and Theorem
\ref{1.2.} remain true.
\end{remark}

The proof is based on variational method. The main difficulties
lie in the appearance of the non-local term and
the lack of compactness due to the unboundedness of the domain
$\R^3$ and the nonlinearity with the critical Sobolev growth. As we
will see later,  the competing effect of the nonlocal term with the
nonlinearity $f(u)$ and the lack of compactness of the embedding
prevent us from using the variational methods in a standard way.

To complete this section,  we  outline the sketch of our proof.

Define $f(s) = 0$ for $s \le 0$. We will work with the following
equation equivalent to $({{E_\varepsilon }})$
\[
\left\{ \begin{gathered}
   - \Bigl( {a + b\int_{{\R^3}} {{{| {\nabla u} |}^2}} } \Bigr)\Delta u
    + V({\varepsilon z})u = f(u) + {u^5}{\text{ in }}{\R^3}, \hfill \\
  u \in {H^1}({{\R^3}}),{\text{ }}u > 0{\text{ in }}{\R^3}. \hfill \\
\end{gathered}  \right.\tag{${{{\hat E}_\varepsilon }}$}
\]

The energy functional corresponding to $({{{\hat E}_\varepsilon }})$ is
\[
{I_\varepsilon }(u) = \frac{a} {2}\int_{{\R^3}} {{{| {\nabla u}
|}^2}}  + \frac{1} {2}\int_{{\R^3}} {V({\varepsilon z}){u^2}}  +
\frac{b} {4}{\Bigl( {\int_{{\R^3}} {{{| {\nabla u} |}^2}} }
\Bigr)^2} - \int_{{\R^3}} {\Bigl( {F(u) + \frac{1} {6}{{( {{u^ +
}})}^6}} \Bigr)}, {\text{ }}u \in {H_\varepsilon },
\]
where ${H_\varepsilon }: = \{ {u \in {H^1}({\R^3})| {\int_{{\R^3}}
{V(\varepsilon z){u^2}}  < \infty } .} \}$ endowed with the norm
\begin{equation}\label{1.8}
{\| u \|_\varepsilon } = {\Bigl( {a\int_{{\R^3}} {{{| {\nabla u}
|}^2}} + \int_{{\R^3}} {V({\varepsilon z}){u^2}} } \Bigr)^{\frac{1}
{2}}}.
\end{equation}

Unlike \cite{hz2} and \cite{wtxz}, where the minimum of $V(x)$ is
global and the mountain-pass lemma  can be used globally, here in
the present paper, the condition $(V_2)$ is local, hence we need to
use a local mountain-pass argument introduced in \cite{pf}, which
also helps us to overcome the obstacle caused by the non-compactness
due to the unboundedness of the domain. To this end, we should
modify the nonlinear terms.

For the bounded domain $\Lambda $ given in $(V_2)$, $k > 2$, $a' >
0$
 such that $f({a'}) + {({a'})^5} = \frac{\alpha }
{k}a'$ where $\alpha$ is mentioned in $(V_1)$, we consider a new problem
\[
 - \Bigl( {a + b\int_{{\R^3}} {{{| {\nabla u} |}^2}} } \Bigr)\Delta u
 + V({\varepsilon z})u = g({\varepsilon z,u}){\text{ in }}{\R^3}\tag{${{\hat E'}_\varepsilon }$}
\]
where
\[
g({z,s}) = \chi(z)\Bigl( {f(s) + {{({{s^ + }})}^5}} \Bigr) +
( {1 - \chi (z)})\tilde f(s)
\]
with
\[
\tilde f(s) = \left\{ \begin{gathered}
  f(s) + {({{s^ + }})^5}{\text{ if  }}s \le a', \hfill \\
  \frac{\alpha }
{k}s{\text{ if  }}s > a' \hfill \\
\end{gathered}  \right.
\]
and $\chi(z)$ is a smooth function such that $\chi(z) = 1$ on
$\Lambda $, $0 \le \chi (z) \le 1$ on $\Lambda '\backslash \Lambda
$, $\chi(z) = 0$ on ${\R^3}\backslash \Lambda '$, where ${\Lambda
'}$ is a suitable open set satisfying $\bar \Lambda  \subset \Lambda
'$ and $V(z) > \mathop {\inf }\limits_{\xi  \in \Lambda } V(\xi)$
for all $z \in \overline {\Lambda '} \backslash \Lambda $.  It is
easy to see that under the assumptions $(f_1)$-$(f_4)$, $g({z,s})$
is a Caratheodory function
and satisfies the following assumptions:\\
$(g_1)$ $g({z,s}) = o({{s^3}})$ near $s  = 0$ uniformly on $z \in {\R^3}$;\\
$(g_2)$ $g({z,s}) \le f(s) + {({{s^ + }})^5}$;\\
$(g_3)$ $0 < 4 G({z,s}) \le g({z,s})s$ for all
$z \in \Lambda ,{\text{ }}s > 0$ or $z \in {\R^3}\backslash \Lambda ,{\text{ }}s \le a'$;\\
$(g_4)$ $0 < 2\tilde F(s) \le \tilde f(s)s \le \frac{1}
{k}\alpha {s^2} \le \frac{1}
{k}V(z){s^2}$ for all $s > 0$ with the number $k$ satisfying
 $k > 2$, where $\tilde F(s) = \int_0^s {\tilde f(\tau)} d\tau $.\\
In particular, $0 < 2G({z,s}) \le g({z,s})s \le \frac{1}
{k}V(z){s^2}$ for all $z \in {\R^3}\backslash \Lambda ',s > 0$,
where $G({z,s }) = \int_0^s  {g({z,\tau })} d\tau $.\\

The energy functional corresponding to $({{{\hat E'}_\varepsilon }})$ is
\begin{equation}\label{1.9}
{J_\varepsilon }(u) = \frac{a} {2}\int_{{\R^3}} {{{| {\nabla u}
|}^2}}  + \frac{1} {2}\int_{{\R^3}} {V({\varepsilon z}){u^2}}  +
\frac{b} {4}{\Bigl( {\int_{{\R^3}} {{{| {\nabla u} |}^2}} }
\Bigr)^2} - \int_{{\R^3}} {G({\varepsilon z,u})} ,{\text{ }}u \in
{H_\varepsilon }.
\end{equation}
Using a standard method, we can prove that ${J_\varepsilon }$
possesses a mountain-pass energy $c_\varepsilon$. To deal with the
difficulty caused by the non-compactness due to the the critical
growth, we should  estimate precisely   the value of $c_\varepsilon$
and give a threshold value (see Lemma \ref{2.5.} below) under which the $(P.S.)_{c_\varepsilon}$
condition for $J_\varepsilon$ is satisfied. Moreover, to verify the
critical point $v_\varepsilon$ of $J_\varepsilon$ at the level
$c_\varepsilon$ is indeed a solution of the original problem
($E_\varepsilon$), we need to establish a uniform estimate on
$L^\infty$-norm of $v_\varepsilon$ (with respect to $\varepsilon$)
by using the idea introduced by Li in \cite{l1}. We should point out
that the non-local term makes it much more complicated to estimate
the threshold value.

The proof of Theorem \ref{1.2.} is mainly based on
Ljusternik-Schnirelmann theory (see \cite{hz2,ps}, for example).
Firstly, we apply the penalization method  to modify the
nonlinearity $f(u) + {({{u^ + }})^5}$ such that the energy
functional of the modified problem satisfies the $(P.S.)$ condition on
an appropriate manifold. Secondly, using the technique due to Benci
and Cerami \cite{bc}, we establish a relationship between the
category of the set $M$ and the number of solutions for the modified
problem. Finally, we prove that, for $\varepsilon > 0$ small, the
solutions for the modified problem are in fact solutions for the
original problem.

Summarily, the novelty of our results lies in two aspects. Firstly,
differently from \cite{hz2} and \cite{wtxz}, where only the  ground
states concentrating at the global minimum point of $V(z)$ were
obtained, we can construct a bound state which concentrates
exactly at one point of any prescribed set consisting of local minimum
points  of $V(z)$. Hence the solutions obtained in
Theorem~\ref{1.1.} may not be the ground state solution.  Secondly,
we obtain the precise threshold value under which the $(P.S.)$
condition for $J_\varepsilon$ is satisfied. So we can get rid of the
large factor $\lambda$ of $f(u)$ in \cite{wtxz}.

This paper is organized as follows, in Section 2, we give some
preliminary results and obtain a $(P.S.)$ sequence. In Section 3, we
will prove that the $(P.S.)$ sequence will converge  in
$H_\varepsilon$ to a solution of ($E_\varepsilon$), which can
complete the proof of Theorem~\ref{1.1.}. In Section 4, we will use
the Ljusternik-Schnirelmann theory
to prove Theorem ~\ref{1.2.} \\


\section{Preliminaries}

\setcounter{equation}{0}

Taking $\varepsilon  = 1$ for simplicity, we consider the equation
\begin{equation}\label{2.2}
 - \Bigl( {a + b\int_{{\R^3}} {{{| {\nabla u} |}^2}} } \Bigr)\Delta u
 + V(z)u = g({z,u}){\text{ in }}{\R^3}.
\end{equation}

The energy functional associated to \eqref{2.2} is given by
\[
J(u) = \frac{a} {2}\int_{{\R^3}} {{{| {\nabla u} |}^2}}  + \frac{1}
{2}\int_{{\R^3}} {V(z){u^2}}  + \frac{b} {4}{\Bigl( {\int_{{\R^3}}
{{{| {\nabla u} |}^2}} } \Bigr)^2} - \int_{{\R^3}} {G({z,u})}
,{\text{ }}u \in H
\]
and $J \in {C^1}( {H,\R} )$.\\

Clearly $J$ possesses the mountain-pass geometry construct  i.e.
$\exists \,e \in H$, $r > 0$, such that ${\| e \|_H} > r$ and
\[
\mathop {\inf }\limits_{{{\| u \|}_H} = r} J(u) > J(0) \ge J(e).
\]

Hence, by the mountain pass theorem without $(P.S.)$ condition (see
\cite{ar}), we obtain a sequence $\{{{u_n}}\}$ such that
\begin{equation}\label{2.3}
J({{u_n}}) \to c > 0,{\text{ }}J'({{u_n}}) \to 0,{\text{ as }}n \to
\infty,
\end{equation}
where $c$ is the minimax level of
functional $J$ given by
\[
c = \mathop {\inf }\limits_{\gamma  \in \Gamma } \mathop {\sup
}\limits_{t \in [ {0,1} ]} J( {\gamma ( t )} ).
\]
Here $\Gamma  = \{ {\gamma  \in C( {[ {0,1} ],H} )| {\gamma ( 0 ) =
0{\text{ and }}J( {\gamma ( 1 )} ) < 0} } \}$.

Moreover, as in \cite{dn,pf,r}, we  can prove

\[
c = \mathop {\inf }\limits_{u \in H,u \ne 0} \mathop {\sup }
\limits_{\tau  \ge 0} J({\tau u})
= \mathop {\inf }\limits_{u \in H\backslash \{ 0 \},\langle {J'(u),u} \rangle  = 0} J(u) > 0.
\]

For the constant $c$, we have the following estimate
\begin{lemma} \label{2.5.}
\[
c < \frac{1}
{4}ab{S^3} + \frac{1}
{{24}}{b^3}{S^6} + \frac{1}
{{24}}{\Bigl( {{b^2}{S^4} + 4aS} \Bigr)^{\frac{3}
{2}}},
\]
where $S$ is the best Sobolev constant for the embedding ${D^{1,2}}(
{{\R^3}} ) \hookrightarrow {L^6}( {{\R^3}})$.
\end{lemma}

\begin{proof}
Without loss of generalities, we  assume that $0 \in \Lambda $.
Choose $R>0$ such that ${B_{2R}}( 0 ) \subset \Lambda $ and $\varphi
\in C_0^\infty ( B_{2R}(0))$ satisfying $\varphi\equiv 1 $ on
${B_R}(0)$ and $0 \le \varphi  \le 1$ on ${B_{2R}}(0)$.

Given $\delta > 0$, we set
${\psi _\delta }(z): = \varphi (z){w_\delta }(z)$, where
\[{w_\delta }( z ) = {( {3\delta } )^{\frac{1}
{4}}}\frac{1} {{{{( {\delta  + {{| z |}^2}} )}^{\frac{1} {2}}}}}\]
satisfies
\begin{equation}\label{2.4}
\int_{{\R^3}} {{{| {\nabla {w_\delta }} |}^2}} = \int_{{\R^3}} {{{|
{{w_\delta }} |}^6}}  = {S^{\frac{3} {2}}}.
\end{equation}
We see
\begin{equation}\label{2.5}
\int_{{\R^3}\backslash {B_R}(0)} {{{| {\nabla {\psi _\delta }}
|}^2}} = O( {{\delta ^{\frac{1} {2}}}}),{\text{  as }}\delta \to 0.
\end{equation}
Let ${X_\delta }: = \int_{{\R^3}} {{{| {\nabla {v_\delta }} |}^2}}
$, where ${v_\delta }: = {\psi _\delta }/{(\int_{{B_{2R}}(0)}
{|{\psi _\delta }{|^6}} )^{\frac{1} {6}}}$. We find
\begin{equation}\label{2.6}
{X_\delta } \le S + O\Bigl( {{\delta ^{\frac{1}
{2}}}} \Bigr){\text{ as }}\delta  \to 0.
\end{equation}

There exists ${t_\delta } > 0$ such that $\mathop {\sup }\limits_{t
\ge 0} J({t{v_\delta }}) = J({{t_\delta }{v_\delta }})$. Hence
$\frac{{dJ({t{v_\delta }})}} {{dt}}| {_{t = {t_\delta }}}  = 0$,
that is
\[
a{t_\delta }\int_{{\R^3}} {{{| {\nabla {v_\delta }} |}^2}}
 + {t_\delta }\int_{{\R^3}} {V(z){{| {{v_\delta }} |}^2}}
 + bt_\delta ^3{\Bigl( {\int_{{\R^3}} {{{| {\nabla {v_\delta }} |}^2}} } \Bigr)^2}
 - \int_{{\R^3}} {\Bigl( {f({{t_\delta }{v_\delta }}) + {{({{t_\delta }{v_\delta }})}^5}} \Bigr){v_\delta }}  = 0,
\]
which implies
\[
t_\delta ^4 - bX_\delta ^2t_\delta ^2 - \Bigl( {a{X_\delta } +
\int_{{\R^3}} {V(z){{| {{v_\delta }} |}^2}} } \Bigr) \le 0.
\]
Hence
\[
0 \le t_\delta ^2 \le \frac{{bX_\delta ^2 + {{\Bigl( {{b^2}X_\delta
^4 + 4\Bigl( {a{X_\delta } + \int_{{\R^3}} {V(z){{| {{v_\delta }}
|}^2}} } \Bigr)} \Bigr)}^{\frac{1} {2}}}}} {2}: = {T_0}.
\]
Denote ${c_1} = bX_\delta ^2$, ${c_2} = a{X_\delta } + \int_{{\R^3}}
{V(z){{| {{v_\delta }} |}^2}} $, then
\begin{eqnarray*}
  &&J( {{t_\delta }{v_\delta }} ) \\
&\le& \frac{1} {2}t_\delta ^2{\int_{{\R^3}}\Bigl(  {a{{| {\nabla
{v_\delta }} |}^2}}+  {V(z){{| {{v_\delta }} |}^2}} } \Bigr) +
\frac{b} {4}t_\delta ^4{\Bigl( {\int_{{\R^3}} {{{| {\nabla {v_\delta
}} |}^2}} } \Bigr)^2} - \frac{1}
{6}t_\delta ^6 - C\lambda\int_{{\R^3}} {{{({{t_\delta }{v_\delta }})}^{{q_1} + 1}}} \\
   &\le& \frac{1}
{2}{T_0} {\int_{{\R^3}} \Bigl({a{{| {\nabla {v_\delta }} |}^2}} +
{V(z){{| {{v_\delta }} |}^2}} } \Bigr) + \frac{b} {4}T_0^2{\Bigl(
{\int_{{\R^3}} {{{| {\nabla {v_\delta }} |}^2}} } \Bigr)^2} -
\frac{1}
{6}T_0^3 - C\lambda\int_{{\R^3}} {{{( {{t_\delta }{v_\delta }} )}^{{q_1} + 1}}}  \\
   &=& \frac{1}
{2}{T_0}\Bigl( {a{X_\delta } + \int_{{\R^3}} {V(z){{| {{v_\delta }}
|}^2}} } \Bigr) + \frac{1} {4}T_0^2bX_\delta ^2 - \frac{1}
{6}T_0^3 - C\lambda\int_{{\R^3}} {{{( {{t_\delta }{v_\delta }} )}^{{q_1} + 1}}}   \\
   &=& \frac{1}
{{24}}{( {c_1^2 + 4{c_2}} )^{\frac{3}
{2}}} + \frac{1}
{{24}}c_1^3 + \frac{1}
{4}{c_1}{c_2} - C\lambda\int_{{\R^3}} {{{( {{t_\delta }{v_\delta }} )}^{{q_1} + 1}}} . \hfill \\
\end{eqnarray*}
Noting \eqref{2.6} and inequality
\[
{( {a + b} )^\alpha } \le {a^\alpha } + \alpha {( {a + b} )^{\alpha
- 1}}b,\quad \alpha\ge 1,\quad a\,b>0,
\]
 we conclude that
\begin{equation}\label{2.7}
\begin{gathered}
 \quad J({{t_\delta }{v_\delta }})  \hfill \\
   \le \frac{1}
{{24}}{( {{b^2}X_\delta ^4 + 4a{X_\delta }} )^{\frac{3} {2}}} +
C\int_{{\R^3}} {V(z)v_\delta ^2} + \frac{1} {{24}}{b^3}X_\delta ^6 +
\frac{1}
{4}abX_\delta ^3 + C\int_{{\R^3}} {V(z)v_\delta ^2}  \hfill \\
  {\text{ }} \quad- C\lambda t_\delta ^{{q_1} + 1}\int_{{\R^3}} {v_\delta ^{{q_1} + 1}}  \hfill \\
   \le \frac{1}
{{24}}{\Bigl( {{b^2}{{\Bigl( {S + O( {{\delta ^{\frac{1}
{2}}}} )} \Bigr)}^4} + 4a\Bigl( {S + O( {{\delta ^{\frac{1}
{2}}}} )} \Bigr)} \Bigr)^{\frac{3}
{2}}} + \frac{1}
{{24}}{b^3}{\Bigl( {S + O( {{\delta ^{\frac{1}
{2}}}} )} \Bigr)^6} \hfill \\
  {\text{ }} \quad+ \frac{1}
{4}ab{\Bigl( {S + O( {{\delta ^{\frac{1} {2}}}} )} \Bigr)^3} +
C\int_{{\R^3}} {V(z)v_\delta ^2}
- C\lambda t_\delta ^{{q_1} + 1}\int_{{\R^3}} {v_\delta ^{{q_1} + 1}}  \hfill \\
   \le \frac{1}
{{24}}{( {{b^2}{S^4} + 4aS} )^{\frac{3}
{2}}} + \frac{1}
{{24}}{b^3}{S^6} + \frac{1}
{4}ab{S^3} + O({{\delta ^{\frac{1}
{2}}}}) \hfill \\
\quad+ \int_{{\R^3}} {({CV(z)v_\delta ^2
- C\lambda t_\delta ^{{q_1} + 1}v_\delta ^{{q_1} + 1}})} . \hfill \\
\end{gathered}
\end{equation}

We can assume that there is a positive constant
${t_0}$ such that ${t_\delta } \ge {t_0} > 0$,
$\forall \delta  > 0$. Otherwise, we could find a
sequence ${\delta _n} \to 0$ as $n \to \infty $
such that ${t_{{\delta _n}}} \to 0$ as $n \to \infty $.
Now up to a subsequence, we have
${t_{{\delta _n}}}{v_{{\delta _n}}} \to 0$ in $H$ as $n \to \infty $. Therefore
\[
0 < c \le \mathop {\sup }\limits_{t \ge 0} J({t{v_{{\delta _n}}}})
= J( {{t_{{\delta _n}}}{v_{{\delta _n}}}}) \to J(0) = 0,
\]
which is a contradiction.

From \eqref{2.7}, to complete the proof,  it suffices to show that
\begin{equation}\label{2.8}
\mathop {\lim }\limits_{\delta  \to {0^ + }} \frac{1}
{{{\delta ^{\frac{1}
{2}}}}}\int_{{B_R}(0)} {({CV(z)v_\delta ^2
- C\lambda t_\delta ^{{q_1} + 1}v_\delta ^{{q_1} + 1}})}
=  - \infty
\end{equation}
and
\begin{equation}\label{2.9}
\mathop {\lim }\limits_{\delta  \to {0^ + }} \frac{1} {{{\delta
^{\frac{1} {2}}}}}\int_{{\R^3}\backslash {B_R}(0)} {( {CV(z)v_\delta
^2 - C\lambda t_\delta ^{{q_1} + 1}v_\delta ^{{q_1} + 1}} )}  \le C.
\end{equation}
In fact,
\[\begin{gathered}
  \frac{1}
{{{\delta ^{\frac{1} {2}}}}}\int_{{B_R}(0)} {CV(z)v_\delta ^2} \le
\frac{C} {{{\delta ^{\frac{1} {2}}}}}\int_{{B_R}(0)} \frac{\delta
^{\frac{1} {2}}}{ {{\delta  + {{| z |}^2}}}} \le C(R)
\end{gathered}
\]
and
\begin{eqnarray*}
  \frac{1}
{{{\delta ^{\frac{1} {2}}}}}\lambda\int_{{B_R}(0)} {t_\delta ^{{q_1}
+ 1}v_\delta ^{{q_1} + 1}}
  & \ge& \frac{C\lambda}
{{{\delta ^{\frac{1} {2}}}}}\int_{{B_R}(0)} {w_\delta ^{{q_1} + 1}}
= \frac{C\lambda} {\delta ^{\frac{1} {2}}}\int_{B_R(0)} \frac{\delta
^{\frac{{{q_1} + 1}} {4}}} {( \delta  + | z
|^2 )^{\frac{q_1 + 1} {2}}}\\
&\ge& C\lambda{\delta ^{\frac{{ - {q_1} + 3}} {4}}}.
\end{eqnarray*}
If $3 < {q_1} < 5$,  \eqref{2.8} holds, while if ${q_1} = 3$, we
choose $\lambda  = 1/\delta$,  \eqref{2.8} also holds.

 Since
\[\begin{gathered}
  \quad\frac{1}
{{{\delta ^{\frac{1} {2}}}}}\int_{{\R^3}\backslash {B_R}(0)} {(
{CV(z)v_\delta ^2 - C\lambda t_\delta ^{{q_1} + 1}v_\delta ^{{q_1} +
1}} )}
   \le \frac{1}
{{{\delta ^{\frac{1}
{2}}}}}\int_{{B_{2R}}(0)\backslash {B_R}(0)} {CV(z)v_\delta ^2}  \le C(R), \hfill \\
\end{gathered} \]
then \eqref{2.9} holds.
\end{proof}

\begin{lemma}\label{2.6.}
Every sequence $\{ {{u_n}} \}$ satisfying
\eqref{2.3} is bounded in $H$.
\end{lemma}

\begin{proof}
Observing $(g_3)$ and $(g_4)$,  we have
\[\begin{gathered}
  \quad J({{u_n}}) - \frac{1}
{4}\langle {J'({{u_n}}),{u_n}} \rangle  \hfill \\
   = \frac{1}
{4}\| {{u_n}} \|_H^2 + \frac{1}
{4}\int_{{\R^3}} {( {g( {z,{u_n}} ){u_n} - 4G( {z,{u_n}} )} )}  \hfill \\
   \ge \frac{1}
{4}\| {{u_n}} \|_H^2 + \frac{1}
{4}\int_{\Lambda '\backslash \Lambda } {({1 - \chi })( {\tilde f({{u_n}}){u_n} -
 4\tilde F( {{u_n}} )} )}  + \frac{1}
{4}\int_{{\R^3}\backslash \Lambda '} {( {g( {z,{u_n}} ){u_n} - 4G( {z,{u_n}} )} )}  \hfill \\
   \ge \frac{1}
{4}\| {{u_n}} \|_H^2 - \frac{1}
{2}\int_{\Lambda '\backslash \Lambda } {( {1 - \chi } )\tilde F( {{u_n}} )}  - \frac{1}
{2}\int_{{\R^3}\backslash \Lambda '} {G( {z,{u_n}} )}  \hfill \\
   \ge \frac{1}
{4}\| {{u_n}} \|_H^2 - \frac{1}
{{4k}}\int_{\Lambda '\backslash \Lambda } {( {1 - \chi } )V(z)u_n^2}  - \frac{1}
{{4k}}\int_{{\R^3}\backslash \Lambda '} {V(z)u_n^2}  \hfill \\
   \ge \frac{1}
{4}\Bigl( {1 - \frac{1}
{k}} \Bigr)\| {{u_n}} \|_H^2. \hfill \\
\end{gathered} \]
By the choice of $k$, we get the upper bound of ${\| {{u_n}} \|_H}$.
\end{proof}

\begin{lemma}\label{2.7.}
There is a sequence $\{ {{z_n}} \} \subset {\R^3}$ and $R > 0$,
$\beta > 0$ such that
\[
\int_{{B_R}({{z_n}})} {u_n^2}  \ge \beta ,
\]
where $\{ {u_n} \}$ is the sequence given by Lemma \ref{2.6.}.
\end{lemma}

\begin{proof}
Suppose by contradiction that the lemma does not hold.
Then by the Vanishing Theorem (Lemma 1.1 of \cite{l2}) it follows that
\[
\int_{{\R^3}} {{{| {{u_n}} |}^s}}  \to 0,{\text{ as }}n \to \infty
{\text{ for all }}2 < s < 6
\]
and then
\[
\int_{{\R^3}} {F( {{u_n}} )}  \to 0,{\text{ }}\int_{{\R^3}}
{f({{u_n}}){u_n}} \to 0{\text{ as }}n \to \infty .
\]
This implies that
\begin{equation}\label{2.10}
\begin{array}{ll}
  \dis\int_{{\R^3}} {G({z,{u_n}})}  \le & \dis\frac{1}
{6}\int_{\Lambda  \cup \{ {z| { {{u_n} \le a'} \}} } } {{{( {u_n^ +
} )}^6}} +\dis \frac{1}
{6}\int_{( {\Lambda '\backslash \Lambda } ) \cap \{ {z| { {{u_n} > a'} \}} } } {\chi {{( {u_n^ + } )}^6}}  \vspace{0.2cm}\\
  &+ \dis\frac{\alpha }
{{2k}}\dis\int_{( {{\R^3}\backslash \Lambda '} ) \cap \{ {z| {
{{u_n}
> a'} \}} } } {u_n^2} + \dis\frac{\alpha } {{2k}}\int_{( {\Lambda
'\backslash \Lambda } ) \cap \{ {z| { {{u_n} > a'} \}} } } {( {1 -
\chi } )u_n^2} + o(1)
\end{array}
\end{equation}
and
\begin{equation}\label{2.11}
\begin{array}{ll}
  \dis\int_{{\R^3}} {g( {z,{u_n}} ){u_n}}  = &
  \dis\int_{\Lambda  \cup \{ {z| { {{u_n} \le a'} \}} } } {{{( {u_n^ + } )}^6}}
  + \dis\int_{( {\Lambda '\backslash \Lambda } ) \cap \{ {z| { {{u_n} > a'} \}} } } {\chi {{( {u_n^ + } )}^6}}   \vspace{0.2cm}\\
  &+ \dis\frac{\alpha }
{k}\int_{( {{\R^3}\backslash \Lambda '} ) \cap \{ {z| { {{u_n} > a'}
\}} } } {u_n^2} + \dis\frac{\alpha } {k}\int_{( {\Lambda '\backslash
\Lambda } ) \cap \{ {z| { {{u_n} > a'} \}} } } {( {1 - \chi }
)u_n^2}
+ o(1). \\
\end{array}
\end{equation}
Hence, using $\langle {J'( {{u_n}} ),{u_n}} \rangle  = o( 1 )$, we
conclude that
\begin{equation}\label{2.12}
\begin{gathered}
  \quad\| {{u_n}} \|_H^2 - \frac{\alpha }
{k}\int_{( {{\R^3}\backslash \Lambda '} ) \cap \{ {z| { {{u_n} > a'}
\}} } } {u_n^2} - \frac{\alpha } {k}\int_{( {\Lambda '\backslash
\Lambda } ) \cap \{ {z| { {{u_n} > a'} \}} } } {( {1 - \chi }
)u_n^2}
+ b{\Bigl( {\int_{{\R^3}} {{{| {\nabla {u_n}} |}^2}} } \Bigr)^2} \hfill \\
   = \int_{\Lambda  \cup \{ {z| { {{u_n} \le a'} \}} } } {{{( {u_n^ + } )}^6}}
   + \int_{( {\Lambda '\backslash \Lambda } ) \cap \{ {z| { {{u_n} > a'} \}} } } {\chi {{( {u_n^ + } )}^6}}  + o(1). \hfill \\
\end{gathered}
\end{equation}
Let ${l_1} \ge 0$, ${l_2} \ge 0$ be such that
\begin{equation}\label{2.13}
\| {{u_n}} \|_H^2 - \frac{\alpha } {k}\int_{( {{\R^3}\backslash
\Lambda '} ) \cap \{ {z| { {{u_n} > a} \}} } } {u_n^2}  -
\frac{\alpha } {k}\int_{( {\Lambda '\backslash \Lambda } ) \cap \{
{z| { {{u_n} > a} \}} } } {( {1 - \chi } )u_n^2}  \to {l_1}
\end{equation}
and
\begin{equation}\label{2.14}
b{\Bigl( {\int_{{\R^3}} {{{| {\nabla {u_n}} |}^2}} } \Bigr)^2} \to
{l_2}{\text{ as }}n \to \infty .
\end{equation}
It is easy to check that ${l_1} > 0$, otherwise $\|{u_n}\|_{H} \to
0$
 as $n \to \infty $ which contradicts $c > 0$.
 From \eqref{2.12}, \eqref{2.13}, \eqref{2.14}, we get
\begin{equation}\label{2.15}
\int_{\Lambda  \cup \{ {z| { {{u_n} \le a} \}} } } {{{( {u_n^ + } )}^6}}
 + \int_{({\Lambda '\backslash \Lambda }) \cap \{ {z| { {{u_n} > a} \}} } } {\chi {{( {u_n^ + } )}^6}}  \to {l_1} + {l_2}.
\end{equation}
By \eqref{2.10}, \eqref{2.13}, \eqref{2.15} and $J({{u_n}}) = c + o(1)$ we have
\begin{equation}\label{2.16}
c \ge \frac{1}
{3}{l_1} + \frac{1}
{{12}}{l_2}.
\end{equation}
Now, using the definition of the constant $S$, we have
\[\begin{gathered}
  \quad \| {{u_n}} \|_H^2 - \frac{\alpha }
{k}\int_{( {{\R^3}\backslash \Lambda '} ) \cap \{ {z| { {{u_n} > a'}
\}} } } {u_n^2}  - \frac{\alpha } {k}\int_{({\Lambda '\backslash
\Lambda }) \cap
\{ {z| { {{u_n} > a'} \}} } } {({1 - \chi })u_n^2}  \hfill \\
   \ge aS{\Bigl( {\int_{\Lambda  \cup \{ {z| { {{u_n} \le a'} \}} } } {{{( {u_n^ + } )}^6}}  +
   \int_{( {\Lambda '\backslash \Lambda })
   \cap \{ {z| { {{u_n} > a'} \}} } } {\chi {{( {u_n^ + } )}^6}} } \Bigr)^{\frac{1}
{3}}} \hfill \\
\end{gathered} \]
and
\[
b{\Bigl( {\int_{{\R^3}} {{{| {\nabla {u_n}} |}^2}} } \Bigr)^2} \ge
b{S^2}{\Bigl( {\int_{\Lambda  \cup \{ {z| { {{u_n} \le a'} \}} } }
{{{( {u_n^ + } )}^6}} + \int_{( {\Lambda '\backslash \Lambda } )
\cap \{ {z| { {{u_n} > a'} \}} } } {\chi {{( {u_n^ + } )}^6}} }
\Bigr)^{\frac{2} {3}}}.
\]
Taking the limit in the above two inequalities, as $n \to \infty $, we achieve that
\[
{l_1} \ge aS{( {{l_1} + {l_2}} )^{\frac{1}
{3}}}
\]
and
\[
{l_2} \ge b{S^2}{( {{l_1} + {l_2}} )^{\frac{2}
{3}}}.
\]
Hence
\[
{( {{l_1} + {l_2}} )^{\frac{1}
{3}}} \ge \frac{{b{S^2} + {{( {{b^2}{S^4} + 4aS} )}^{\frac{1}
{2}}}}}
{2}
\]
and
\[\begin{gathered}
   c \ge \frac{1}
{3}{l_1} + \frac{1} {{12}}{l_2}
   \ge \frac{1}
{3}aS{( {{l_1} + {l_2}} )^{\frac{1}
{3}}} + \frac{1}
{{12}}b{S^2}{( {{l_1} + {l_2}} )^{\frac{2}
{3}}} \hfill \\
   \ge \frac{1}
{4}ab{S^3} + \frac{1}
{{24}}{b^3}{S^6} + \frac{1}
{{24}}{( {{b^2}{S^4} + 4aS} )^{\frac{3}
{2}}}, \hfill \\
\end{gathered} \]
which contradicts Lemma \ref{2.5.}.
\end{proof}

\begin{lemma}\label{2.8.}
The sequence $\{ {{z_n}} \}$ given in Lemma \ref{2.7.} is bounded in
${\R^3}$.
\end{lemma}

\begin{proof}
For each $\rho  > 0$ consider a smooth cut-off
function $0 \le {\psi _\rho } \le 1$ such that
\[{\psi _\rho }(z) = \left\{ \begin{gathered}
  0{\text{  if }}|z| \le \rho , \hfill \\
  1{\text{  if }}|z| \ge 2\rho,\hfill \\
\end{gathered}  \right.
\quad  \quad |{\nabla {\psi _\rho }}| \le \frac{C} {\rho }.
\]
 Using
$\langle {J'( {{u_n}} ),{\psi _\rho }{u_n}} \rangle = o(1)$, we
obtain
\[\begin{gathered}
  a\int_{{\R^3}} {{{| {\nabla {u_n}} |}^2}{\psi _\rho }}
  + a\int_{{\R^3}} {( {\nabla {u_n} \cdot \nabla {\psi _\rho }} ){u_n}}
  + b\int_{{\R^3}} {{{| {\nabla {u_n}} |}^2}} \Bigl( {\int_{{\R^3}} {{{| {\nabla {u_n}} |}^2}{\psi _\rho } + }
   \int_{{\R^3}} {( {\nabla {u_n} \cdot \nabla {\psi _\rho }} ){u_n}} } \Bigr) \hfill \\
   \quad+ \int_{{\R^3}} {V(z)u_n^2} {\psi _\rho } \hfill\\
   = \int_{{\R^3}} {g( {z,{u_n}} ){u_n}{\psi _\rho }}
   + o(1). \hfill \\
\end{gathered} \]
Choose $\rho $ large enough such that $\Lambda ' \subset {B_\rho
}(0)$, we have
\[\begin{gathered}
  \quad \Bigl( {1 - \frac{1}
{k}} \Bigr)\int_{{\R^3}} {V(z)u_n^2{\psi _\rho }}  \hfill \\
   \le  - a\int_{{\R^3}} {( {\nabla {u_n} \cdot \nabla {\psi _\rho }} ){u_n}}
   - b\int_{{\R^3}} {{{| {\nabla {u_n}} |}^2}} \int_{{\R^3}}
   {( {\nabla {u_n} \cdot \nabla {\psi _\rho }} ){u_n}}  + o(1) \hfill \\
   \le \frac{C}
{\rho }\int_{{\R^3}} {| {\nabla {u_n}} || {{u_n}} |} + \frac{C}
{\rho }\int_{{\R^3}} {{{| {\nabla {u_n}} |}^2}\int_{{\R^3}}
 {| {\nabla {u_n}} || {{u_n}} |} }  + o(1) \hfill \\
   \le \frac{C}
{\rho } + o(1). \hfill \\
\end{gathered} \]
Hence we get
\[
\int_{| z | \ge 2\rho } {u_n^2}  \le \frac{C}
{\rho } + o(1).
\]
If $\{ {{z_n}} \}$ is unbounded, Lemma \ref{2.7.}
and the above estimate give that
\[
0 < \beta  \le \frac{C}
{\rho }
\]
which leads to a contradiction for large $\rho$.
\end{proof}

Using standard argument, up to a subsequence, we
may assume that there is $u \in H$ such that
\begin{equation}\label{2.17}
\left\{ \begin{gathered}
  {u_n} \rightharpoonup u{\text{  in  }}H, \hfill \\
  {u_n} \to u{\text{  in  }}L_{{\text{loc}}}^s( {{\R^3}} )
  {\text{  for all  }}1 \le s < 6, \hfill \\
  {u_n} \to u{\text{  a}}{\text{.e}}{\text{. in  }}{\R^3}. \hfill \\
\end{gathered}  \right.
\end{equation}
By Lemma \ref{2.7.} and Lemma \ref{2.8.}, $u$ is nontrivial.
Moreover, for any $\varphi  \in H$, we get
\begin{equation}\label{2.18}
a\int_{{\R^3}} {\nabla u \cdot \nabla \varphi } + \int_{{\R^3}}
{V(z)u\varphi } + bA\int_{{\R^3}} {\nabla u \cdot \nabla \varphi } -
\int_{{\R^3}} {g({z,u})\varphi }  = 0,
\end{equation}
where $A: = \mathop {\lim }\limits_{n \to \infty } \int_{{\R^3}}
{{{| {\nabla {u_n}} |}^2}}$ and  $\int_{{\R^3}} {{{| {\nabla u}
|}^2}}  \le A$. Taking $\varphi  = u$, we get
\[\langle {J'( u ),u} \rangle  \le 0.\]
Now, we prove that
\begin{equation}\label{2.19}
\langle {J'(u),u} \rangle  = 0.
\end{equation}
Assuming the contrary, if $\langle {J'( u ),u} \rangle  < 0$,
there is a unique $0 < t < 1$ such that
\[
\langle {J'( {tu} ),tu} \rangle  = 0.
\]
So,
\begin{equation}\label{add}
\begin{array}{ll}
  c &\le J({tu}) - \dis\frac{1}
{4}\langle {J'( {tu} ),tu} \rangle  \vspace{0.2cm}\\
   &= \dis\frac{{{t^2}}}
{4}\Bigl( {a\int_{{\R^3}} {{{| {\nabla u} |}^2}} + \dis\int_{{\R^3}}
{V(z){u^2}} } \Bigr) + \int_{{\R^3}} {\Bigl( {\frac{1}
{4}g( {z,tu} )tu - G( {z,tu} )} \Bigr)}  \vspace{0.2cm}\\
  &< \dis \frac{1}
{4}\Bigl( {a\int_{{\R^3}} {{{| {\nabla u} |}^2}} + \int_{{\R^3}}
{V(z){u^2}} } \Bigr) + \int_{{\R^3}} {\Bigl( {\frac{1}
{4}g( {z,u} )u - G( {z,u} )} \Bigr)}  \vspace{0.2cm}\\
   &\le \mathop {\underline {\lim } }\limits_{n \to \infty } \dis\frac{1}
{4}\Bigl( {a\int_{{\R^3}} {{{| {\nabla {u_n}} |}^2}} + \int_{{\R^3}}
{V(z)u_n^2} } \Bigr) + \int_{{\R^3}} {\Bigl( {\frac{1}
{4}g( {z,{u_n}} ){u_n} - G( {z,{u_n}} )} \Bigr)}  \vspace{0.2cm}\\
   &= \mathop {\underline {\lim } }\limits_{n \to \infty }\Bigl \{ {J( {{u_n}} ) - \dis\frac{1}
{4}\Bigl\langle {J'( {{u_n}} ),{u_n}} \Bigr\rangle } \Bigr\} = c,
\end{array}
\end{equation}
which causes a contradiction. Hence, \eqref{2.19} follows and $A =
\int_{{\R^3}} {|\nabla u{|^2}} $. Using \eqref{add} again with
$t=1$, we conclude $J(u) = c$.

Hence, we indeed prove

\begin{proposition}\label{2.9.}
The functional ${J_\varepsilon }$ possesses a nontrivial
critical point ${v_\varepsilon } \in {H_\varepsilon }$ such that
\begin{equation}\label{2.20}
{J_\varepsilon }( {{v_\varepsilon }} )
= \mathop {\inf }\limits_{\gamma  \in \Gamma } \mathop {\sup }\limits_{t \in [ {0,1} ]} {J_\varepsilon }( {\gamma (t)} )
 = \mathop {\inf }\limits_{u \in {H_\varepsilon }\backslash \{ 0 \}} \mathop {\sup }\limits_{\tau  \ge 0} {J_\varepsilon }( {\tau u} )
 = \mathop {\inf }\limits_{u \in {H_\varepsilon }\backslash \{ 0 \},\langle {{{J'}_\varepsilon }(u),u} \rangle=0} {J_\varepsilon }(u),
\end{equation}
where $\Gamma  = \{ {\gamma  \in C( {[ {0,1} ],{H_\varepsilon }} )
| {\gamma ( 0 ) = 0{\text{ and }}{J_\varepsilon }( {\gamma ( 1 )} ) < 0} } \}$.\\
\end{proposition}

Now, we consider the following equation
\begin{equation}\label{2.21}
\left\{ \begin{gathered}
   - \Bigl( {a + b\int_{{\R^3}} {{{| {\nabla u} |}^2}} } \Bigr)
   \Delta u + \overline V u = f(u) + {u^5}{\text{ in }}{\R^3}, \hfill \\
  u \in {H^1}( {{\R^3}} ),{\text{ }}u > 0{\text{ in }}{\R^3} \hfill \\
\end{gathered}  \right.
\end{equation}
where $\overline V $ is a positive constant. The
functional corresponding to \eqref{2.21} is
\[
{I_{\overline V }}(u) = \frac{a} {2}\int_{{\R^3}} {{{| {\nabla u}
|}^2}}  + \frac{1} {2}\int_{{\R^3}} {\overline V {u^2}}  + \frac{b}
{4}{\Bigl( {\int_{{\R^3}} {{{| {\nabla u} |}^2}} } \Bigr)^2}
 - \int_{{\R^3}} {\Bigl( {F(u) + \frac{1}
{6}{{( {{u^ + }} )}^6}} \Bigr)} .
\]

\begin{proposition}\label{2.10.}
Suppose that $f(u)$ satisfies $({f_1})$-$({f_4})$, then \eqref{2.21}
has a positive ground-state solution $w \in {H^1}( {{\R^3}} ) \cap
C_{{\text{loc}}}^{2,\alpha }( {{\R^3}} )$, such that ${I_{\overline
V }}(w) = {c_{\overline V }} > 0$, where
\[
{c_{\overline V }} = \mathop {\inf }\limits_{{\mathcal{N}_{\overline V }}} {I_{\overline V }}(u)
\]
and
\[
{\mathcal{N}_{\overline V }} = \{ {u \in {H^1}( {{\R^3}} )| {u \ne
0,\langle {{{I'}_{\overline V }}( u ),u} \rangle  = 0} } \}
\]
is the Nehari manifold of ${I_{\overline V }}$. Moreover,
${I_{\overline V }} (w) = \mathop {\inf }\limits_{u \in {H^1}(
{{\R^3}} )\backslash \{ 0 \}} \mathop {\sup }\limits_{\tau  \ge 0}
{I_{\overline V }} ( {\tau u} )$.
\end{proposition}

\begin{proof}
Similar to the proof of Proposition \ref{2.9.}, we can  get the existence
of a $w \in {H^1}( {{\R^3}} )$ such that ${{I'}_{\bar V}}(w) = 0$
and ${I_{\overline V }}( w ) = {c_{\overline V }}>0$.  By elliptic
regularity theory, $w \in C_{{\text{loc}}}^{2,\alpha }( {{\R^3}})$.
Since $f(s) = 0$ for $s \le 0$, $w \ge 0$. By the strong maximum
principle, $w > 0$. Similar to Proposition \ref{2.9.}, ${c_{\overline V }}
= \mathop {\inf }\limits_{u \in {H^1}( {{\R^3}} )\backslash \{ 0 \}}
\mathop {\sup }\limits_{\tau \ge 0} {I_{\overline V }}( {\tau u} )$.

\end{proof}

For ${V_0}: = \mathop {\min }\limits_\Lambda  V$,
let $w$ be a ground-state solution to the equation
\begin{equation}\label{2.22}
 - \Bigl( {a + b\int_{{\R^3}} {{{| {\nabla w} |}^2}} } \Bigr)
 \Delta w + {V_0}w = f(w) + {( {{w^ + }})^5}
\end{equation}
satisfying
\begin{equation}\label{2.23}
{I_{{V_0}}}\left( w \right) = \mathop {\inf }\limits_{v \in
{H^1}\left( {{\R^3}} \right)\backslash \left\{ 0 \right\}} \mathop
{\sup }\limits_{\tau  \ge 0} {I_{{V_0}}}\left( {\tau v} \right) : =
{c_{{V_0}}}.
\end{equation}

\begin{lemma}\label{2.11.}
\begin{equation}\label{2.24}
{J_\varepsilon }( {{v_\varepsilon }} ) \le {c_{{V_0}}}  + o(1).
\end{equation}
\end{lemma}

\begin{proof}
The proof is similar to what was done in \cite{pf}. Let ${z_0} \in
\Lambda $ be such that $V( {{z_0}}) = {V_0}$ and $u_\varepsilon(z) =
\eta ( {\frac{{\varepsilon z - {z_0}}} {{\sqrt \varepsilon  }}} )w(
{\frac{{\varepsilon z - {z_0}}} {\varepsilon }} )$ where $\eta $ is
a smooth cut-off function with $0 \le \eta  \le 1$, $\eta  = 1$ on
${B_1}(0)$, $\eta  = 0$ on ${\R^3}\backslash {B_2}(0)$, $| {\nabla
\eta } | \le C$. Since $w > 0$, by the arguments as in the proof of
Lemma \ref{2.5.}, there is a unique ${t_\varepsilon } > 0$ such that
$\mathop {\sup }\limits_{t > 0} {J_\varepsilon }( {tu_\varepsilon} )
= {J_\varepsilon }( {{t_\varepsilon }u_\varepsilon} )$ and
$\frac{{d{J_\varepsilon }( {tu_\varepsilon} )}} {{dt}}| {_{t =
{t_\varepsilon }}}  = 0$, i.e.
\begin{equation}\label{2.25}
a{t_\varepsilon }\int_{{\R^3}} {{{| {\nabla u_\varepsilon} |}^2}} +
{t_\varepsilon }\int_{{\R^3}} {V( {\varepsilon z}){u_\varepsilon^2}}
+ bt_\varepsilon ^3{\Bigl( {\int_{{\R^3}} {{{| {\nabla
u_\varepsilon} |}^2}} } \Bigr)^2}
 - \int_{{\R^3}} {\Bigl( {f( {{t_\varepsilon }u_\varepsilon} )u_\varepsilon
 + t_\varepsilon ^5{u_\varepsilon^6}} \Bigr)}  = 0.
\end{equation}
We claim that there exist ${t_0},{T_0} > 0$ such that $0 < {t_0} \le
{t_\varepsilon } \le {T_0}$ which will be proved later. Let $z' =
\frac{{\varepsilon z - {z_0}}} {\varepsilon }$, we see
\[\begin{gathered}
  a{t_\varepsilon }\int_{{B_{\frac{1}
{{\sqrt \varepsilon  }}}}(0)} {{{| {\nabla w({z'})} |}^2}}
+ {t_\varepsilon }\int_{{B_{\frac{1}
{{\sqrt \varepsilon  }}}}(0)} {V( {\varepsilon z' + {z_0}} ){{( {w( {z'} )} )}^2}}
+ bt_\varepsilon ^3{\Bigl( {\int_{{B_{\frac{1}
{{\sqrt \varepsilon  }}}}(0)} {{{| {\nabla w( {z'} )} |}^2}} } \Bigr)^2} \hfill \\
  \quad - \int_{{B_{\frac{1}
{{\sqrt \varepsilon  }}}}(0)} {f( {{t_\varepsilon }w( {z'} )} )w( {z'} )}
- t_\varepsilon ^5\int_{{B_{\frac{1}
{{\sqrt \varepsilon  }}}}(0)} {{{( {w( {z'} )} )}^6}}  = o(1). \hfill \\
\end{gathered} \]
Since $0 < {t_0} \le {t_\varepsilon } \le {T_0}$, going if
necessary to a subsequence, ${t_\varepsilon } \to T > 0$, then
\begin{equation}\label{2.26}
\begin{gathered}
  aT\int_{{\R^3}} {{{| {\nabla w( {z'} )} |}^2}}
  + T{\int_{{\R^3}} {{V_0}( {w( {z'} )} )} ^2}
  + b{T^3}{\Bigl( {\int_{{\R^3}} {{{t| {\nabla w( {z'} )} |}^2}} } \Bigr)^2} \hfill \\
   \quad- \int_{{\R^3}} {f( {Tw( {z'} )} )w( {z'} )}
   - {T^5}\int_{{\R^3}} {{{( {w( {z'} )} )}^6}}  = 0. \hfill \\
\end{gathered}
\end{equation}
Since $w$ is a weak solution to \eqref{2.22}, we get
\[
( {\frac{1} {{{T^2}}} - 1} )\Bigl( {a\int_{{\R^3}} {{{| {\nabla w}
|}^2}} + \int_{{\R^3}} {{V_0}{w^2}} } \Bigr) = \int_{{\R^3}}
{{w^4}\Bigl( {\frac{{f( {Tw} )}} {{{{( {Tw} )}^3}}} - \frac{{f( w
)}} {{{w^3}}}} \Bigr) + ( {{T^2} - 1} ){w^6}} .
\]
By $(f_2)$, ${t_\varepsilon } \to T = 1$. Direct calculations show that
\[\begin{gathered}
  \quad\mathop {\sup }\limits_{t > 0} {J_\varepsilon }( {tu_\varepsilon} ) \hfill \\
   = {J_\varepsilon }( {{t_\varepsilon }u_\varepsilon} ) \hfill \\
   = \frac{{at_\varepsilon ^2}}
{2}\int_{{\R^3}} {{{| {\nabla w} |}^2}} + \frac{{t_\varepsilon ^2}}
{2}\int_{{\R^3}} {{V_0}{w^2}} + \frac{{bt_\varepsilon ^4}}
{4}{\Bigl( {\int_{{\R^3}} {{{| {\nabla w} |}^2}} } \Bigr)^2} -
\int_{{\R^3}} {F( {{t_\varepsilon }w} )}  - \frac{{t_\varepsilon
^6}}
{6}\int_{{\R^3}} {{w^6}}  + o(1) \hfill \\
   = \frac{a}
{2}\int_{{\R^3}} {{{| {\nabla w} |}^2}}  + \frac{1} {2}\int_{{\R^3}}
{{V_0}{w^2}}  + \frac{b} {4}{\Bigl( {\int_{{\R^3}} {{{| {\nabla w}
|}^2}} } \Bigr)^2} - \int_{{\R^3}} {F(w)}  - \frac{1}
{6}\int_{{\R^3}} {{w^6}}  + o(1) \hfill \\
   = {c_{{V_0}}}+ o(1). \hfill \\
\end{gathered} \]
Thus \eqref{2.24} follows.

At last, we prove the claim that $0 < {t_0} \le {t_\varepsilon } \le
{T_0}$. Assuming the contrary that ${t_\varepsilon } \to 0$, then by
$(f_1)$, $(f_4)$, we get that
\begin{equation}\label{2.28}
\begin{gathered}
  \quad a{t_\varepsilon }\int_{{\R^3}} {{{| {\nabla u_\varepsilon} |}^2}}
  + {t_\varepsilon }\int_{{\R^3}} {V( {\varepsilon z}){u_\varepsilon^2}}
   + bt_\varepsilon ^3{\Bigl( {\int_{{\R^3}} {{{| {\nabla u_\varepsilon} |}^2}} } \Bigr)^2} \hfill \\
   = \int_{{\R^3}} {f( {{t_\varepsilon }u_\varepsilon} )u_\varepsilon}
   + t_\varepsilon ^5\int_{{\R^3}} {{u_\varepsilon^6}}  \le Ct_\varepsilon ^3\int_{{\R^3}} {{u_\varepsilon^4}}
   + Ct_\varepsilon ^5\int_{{\R^3}} {{u_\varepsilon^6}} . \hfill \\
\end{gathered}
\end{equation}
Direct computations yield
\[\begin{gathered}
  \quad {t_\varepsilon }\Bigl( {a\int_{{\R^3}} {{{| {\nabla w} |}^2}}
   + \int_{{\R^3}} {{V_0}{w^2}}  + o(1)} \Bigr)
   + t_\varepsilon ^3\Bigl( {b{{\Bigl( {\int_{{\R^3}} {{{| {\nabla w} |}^2}} } \Bigr)}^2} + o(1)} \Bigr) \hfill \\
   \le Ct_\varepsilon ^3\Bigl( {\int_{{\R^3}} {{w^4}}  + o(1)} \Bigr)
   + Ct_\varepsilon ^5\Bigl( {\int_{{\R^3}} {{w^6}}  + o(1)} \Bigr), \hfill \\
\end{gathered} \]
which leads to a contradiction.

If ${t_\varepsilon } \to \infty $, then
\begin{equation}\label{2.33}
\begin{array}{ll}
&\quad a{t_\varepsilon }\dis\int_{{\R^3}} {{{| {\nabla
u_\varepsilon} |}^2}} + {t_\varepsilon }\dis\int_{{\R^3}} {V(
{\varepsilon z}){u_\varepsilon^2}} + bt_\varepsilon ^3{\Bigl(
{\dis\int_{{\R^3}} {{{| {\nabla u_\varepsilon} |}^2}} } \Bigr)^2}
\vspace{0.2cm}\\
&=\dis \int_{{\R^3}} {f( {{t_\varepsilon }u_\varepsilon}
)u_\varepsilon} + t_\varepsilon ^5\int_{{\R^3}} {{u_\varepsilon^6}}
\ge t_\varepsilon ^5\int_{{\R^3}} {{u_\varepsilon^6}}.
\end{array}
\end{equation}
Hence,
\[
 a\int_{{\R^3}} {{{| {\nabla w} |}^2}}
+ \int_{{\R^3}} {{V_0}{w^2}}   + t_\varepsilon ^2 {b{{\Bigl(
{\int_{{\R^3}} {{{| {\nabla w} |}^2}} } \Bigr)}^2}} \ge
t_\varepsilon ^4\Bigl( {\int_{{\R^3}} {{w^6}}  + o(1)} \Bigr),
\]
which is a contradiction.
\end{proof}

Since $\langle {{{J'}_\varepsilon }( {{v_\varepsilon }} ),
{v_\varepsilon }} \rangle  = 0$, we have, from \eqref{2.24} that
\[\begin{gathered}
  \quad\frac{a}
{2}\int_{{\R^3}} {{{| {\nabla {v_\varepsilon }} |}^2}}  + \frac{1}
{2}\int_{{\R^3}} {V( {\varepsilon z} )v_\varepsilon ^2}  + \frac{b}
{4}{\Bigl( {\int_{{\R^3}} {{{| {\nabla {v_\varepsilon }} |}^2}} } \Bigr)^2} \hfill \\
   \le {c_{{V_0}}} + o(1) + \int_{{\R^3}} {G( {\varepsilon z,{v_\varepsilon }} )}  \hfill \\
   \le C + \int_{\Lambda /\varepsilon } {G( {\varepsilon z,{v_\varepsilon }} )}  + \int_{( {\Lambda '/ \varepsilon })
   \backslash ( {\Lambda /\varepsilon } )} {G( {\varepsilon z,{v_\varepsilon }} )}  + \int_{{\R^3}\backslash ( \Lambda' /\varepsilon )}
   {G( {\varepsilon z,{v_\varepsilon }} )}  \hfill \\
   \le C + \frac{1}
{4}\int_{\Lambda /\varepsilon } {g( {\varepsilon z,{v_\varepsilon }} ){v_\varepsilon }}  + \frac{1}
{4}\int_{( \Lambda '/\varepsilon )\backslash (  \Lambda / \varepsilon )} {\chi ( {\varepsilon z} )( {f( {{v_\varepsilon }} )
+ v_\varepsilon ^5} ){v_\varepsilon }}  \hfill \\
   \quad+ \int_{(\Lambda '/ \varepsilon)\backslash ( (\Lambda / \varepsilon )} {( {1 - \chi ( {\varepsilon z} )} )\frac{1}
{{2k}}V( {\varepsilon z})v_\varepsilon ^2} {\text{  }} +
\int_{{\R^3}\backslash ( \Lambda' / \varepsilon )} {\frac{1}
{{2k}}V( {\varepsilon z} )v_\varepsilon ^2}  \hfill \\
   \le C + \frac{1}
{4}\int_{{\R^3}} {g( {{\varepsilon z},{v_\varepsilon }}
){v_\varepsilon }}  + \frac{1}
{{2k}}\Bigl( {a\int_{{\R^3}} {{{| {\nabla {v_\varepsilon }} |}^2}}  + \int_{{\R^3}} {V( {\varepsilon z})v_\varepsilon ^2} } \Bigr) \hfill \\
   \le C + \frac{1}
{4}\left( {a\int_{{\R^3}} {{{| {\nabla {v_\varepsilon }} |}^2}}  +
\int_{{\R^3}} {V( {\varepsilon z})v_\varepsilon ^2}  + b{{\Bigl(
{\int_{{\R^3}} {{{| {\nabla {v_\varepsilon }} |}^2}} } \Bigr)}^2}}
\right)  \hfill \\
\quad+ \frac{1} {{2k}}\Bigl( {a\int_{{\R^3}} {{{| {\nabla
{v_\varepsilon }} |}^2}}
 + \int_{{\R^3}} {V( {\varepsilon z})v_\varepsilon ^2} } \Bigr), \hfill \\
\end{gathered} \]
which gives that
\begin{equation}\label{2.34}
\Bigl( {\frac{1} {4} - \frac{1} {{2k}}} \Bigr)\Bigl( {a\int_{{\R^3}}
{{{| {\nabla {v_\varepsilon }} |}^2}} + \int_{{\R^3}} {V(
{\varepsilon z})v_\varepsilon ^2} } \Bigr) \le C.
\end{equation}

Consider the following equation
\begin{equation}\label{2.35}
 - \Bigl( {a + b\int_{{\R^3}} {{{| {\nabla u} |}^2}} } \Bigr)\Delta u
  + {V_n}(z)u = {f_n}({z,u}){\text{ in }}{\R^3}
\end{equation}
where $ \{{V_n}\} \,(n=1,\cdots) $ satisfies
\[
{V_n}(z) \ge \alpha  > 0{\text{ for all }}z \in {\R^3},
\]
and  ${f_n}( {z,t} )$ is a Carathedory function such that for any
$\varepsilon  > 0$, there exists ${C_\varepsilon } > 0$ and
\begin{equation}\label{2.36}
| {{f_n}( {z,t} )} | \le \varepsilon | t | + {C_\varepsilon }{| t
|^5},{\text{ }}\forall ( {z,t} ) \in {\R^3} \times \R.
\end{equation}

\begin{lemma}\label{2.12.}
Assume that $v_n$ are weak solutions to \eqref{2.35} satisfying ${
\| {{v_n}} \|_{{H^1} ({{\R^3}})}} \le C$ for $n \in \mathbb{N}$. If
$ \{ {{{ | {{v_n}}  |}^6}}  \}$ is uniformly integrable near $\infty
$, i.e. $\forall \delta  > 0$, $\exists R > 0$, for any $r > R$,
$\int_{{\R^3}\backslash {B_r} (0)} {{{| {{v_n}} |}^6}}  < \delta $,
then
\begin{equation}\label{2.37}
\mathop {\lim }\limits_{| z | \to \infty } {v_n}(z) = 0{\text{ uniformly for }}n.
\end{equation}
\end{lemma}

\begin{proof}
Following \cite{l1}, for any $R > 0$, $0 < r \le \frac{R}
{2}$, let $\eta  \in {C^\infty }( {{R^N}} )$, $0 \le \eta  \le 1$ with
\[
\eta  = \left\{ \begin{gathered}
  1{\text{  if  }}|z| \ge R ,\hfill \\
  0{\text{  if  }}|z| \le R - r ,\hfill \\
\end{gathered}  \right.
\]
$| {\nabla \eta } | \le \frac{C} {r}$. Set ${( {{v_n}} )_L} = \min (
{{v_n},L} )$ where $L>0$. Taking $\bar v = {\eta ^2}{v_n}( {{v_n}}
)_L^{2( {\beta  - 1} )}$ for ${\beta  \ge 1}$ as a test function in
\eqref{2.35}. Considering \eqref{2.36}, we see that for $\forall\,
\varepsilon
> 0$, $\exists\, {C_\varepsilon } > 0$, such that
\[
\Bigl( {a + b\int_{{\R^3}} {{{| {\nabla {v_n}} |}^2}} }
\Bigr)\int_{{\R^3}} {\nabla {v_n}\nabla \bar v}  + \int_{{\R^3}}
{{V_n}(z){v_n}\bar v}  \le \varepsilon \int_{{\R^3}} {{v_n}\bar v} +
{C_\varepsilon }\int_{{\R^3}} {v_n^5\bar v} .
\]
Taking $\varepsilon  = \alpha $, we get
\[
\Bigl( {a + b\int_{{\R^3}} {{{| {\nabla {v_n}} |}^2}} }
\Bigr)\int_{{\R^3}} {\nabla {v_n}\nabla \bar v}  \le C\int_{{\R^3}}
{v_n^5\bar v} .
\]
For simplicity, we denote by ${A_n}: = \Bigl( {a + b\int_{{\R^3}}
{{{| {\nabla {v_n}} |}^2}} } \Bigr)$.
 We rewrite the above inequality as
\[\begin{gathered}
  {A_n}\Bigl( {2\int_{{\R^3}} {( {\nabla {v_n}\cdot\nabla \eta } )\eta {v_n}( {{v_n}} )_L^{2( {\beta  - 1} )}}  + \int_{{\R^3}} {{{| {\nabla {v_n}} |}^2}{\eta ^2}( {{v_n}} )_L^{2( {\beta  - 1} )}} }  \hfill \\
  \quad { + 2( {\beta  - 1} )\int_{{\R^3}} {{{| {\nabla {{( {{v_n}} )}_L}} |}^2}{\eta ^2}( {{v_n}} )_L^{2( {\beta  - 1} )}} } \Bigr) \le C\int_{{\R^3}} {v_n^6{\eta ^2}( {{v_n}} )_L^{2( {\beta  - 1} )}} . \hfill \\
\end{gathered} \]
By Young's inequality ,we have
\[\begin{gathered}
  \quad{A_n}\Bigl( {\int_{{\R^3}} {{{| {\nabla {v_n}} |}^2}{\eta ^2}( {{v_n}} )_L^{2( {\beta  - 1} )}}  + C( {\beta  - 1} )\int_{{\R^3}} {{{| {\nabla {{( {{v_n}} )}_L}} |}^2}{\eta ^2}( {{v_n}} )_L^{2( {\beta  - 1} )}} } \Bigr) \hfill \\
   \le C{A_n}\int_{{\R^3}} {{{| {\nabla \eta } |}^2}v_n^2( {{v_n}} )_L^{2( {\beta  - 1} )}}  + C\int_{{\R^3}} {v_n^6{\eta ^2}( {{v_n}} )_L^{2( {\beta  - 1} )}} . \hfill \\
\end{gathered} \]
It is clear that $a \le {A_n} \le {a^ * }$ for some ${a^ * } > 0$. Therefore We can rewrite the above inequality as
\begin{equation}\label{2.38}
\begin{gathered}
  \quad \int_{{\R^3}} {{{| {\nabla {v_n}} |}^2}{\eta ^2}( {{v_n}} )_L^{2( {\beta  - 1} )}}  + C( {\beta  - 1} )\int_{{\R^3}} {{{| {\nabla {{( {{v_n}} )}_L}} |}^2}{\eta ^2}( {{v_n}} )_L^{2( {\beta  - 1} )}}  \hfill \\
   \le C\int_{{\R^3}} {{{| {\nabla \eta } |}^2}v_n^2( {{v_n}} )_L^{2( {\beta  - 1} )}}  + C\int_{{\R^3}} {v_n^6{\eta ^2}( {{v_n}} )_L^{2( {\beta  - 1} )}} . \hfill \\
\end{gathered}
\end{equation}
Let ${W_L} = \eta {v_n}( {{v_n}} )_L^{( {\beta  - 1} )}$, by Sobolev's inequality and \eqref{2.38}, we have
\begin{equation}\label{2.39}
\begin{gathered}
  \quad\| {{W_L}} \|_{{L^6}}^2 \le C\int_{{\R^3}} {{{| {\nabla {W_L}} |}^2}}  \hfill \\
   \le C\int_{{\R^3}} {{{| {\nabla \eta } |}^2}v_n^2( {{v_n}} )_L^{2( {\beta  - 1} )}}  + C\int_{{\R^3}} {{\eta ^2}{{| {\nabla {v_n}} |}^2}( {{v_n}} )_L^{2( {\beta  - 1} )}}  \hfill \\
  \quad + C{( {\beta  - 1} )^2}\int_{{\R^3}} {{\eta ^2}{{| {\nabla {{( {{v_n}} )}_L}} |}^2}( {{v_n}} )_L^{2( {\beta  - 1} )}}  \hfill \\
   \le C{\beta ^2}\Bigl( {\int_{{\R^3}} {v_n^6{\eta ^p}( {{v_n}} )_L^{2( {\beta  - 1} )}}  + \int_{{\R^3}} {v_n^2{{| {\nabla \eta } |}^2}( {{v_n}} )_L^{2( {\beta  - 1} )}} } \Bigr). \hfill \\
\end{gathered}
\end{equation}

We claim that there exists $ R > 1$, independent of $n$, such that
\begin{equation}\label{2.40}
{v_n}{\text{ is bounded in  }}{L^{18}}\{ {| z | \ge R} \}.
\end{equation}
In fact, let $\beta  = 3$ and use \eqref{2.39}, we have
\[\begin{gathered}
  {\Bigl( {\int_{{\R^3}} {{{( {\eta {v_n}( {{v_n}} )_L^2} )}^6}} } \Bigr)^{\frac{1}
{3}}} \hfill \\
   \le C{\Bigl( {\int_{{\R^3}} {{{( {\eta {v_n}( {{v_n}} )_L^2} )}^6}} } \Bigr)^{\frac{1}
{3}}}{\Bigl( {\int_{| z | \ge R - r} {v_n^6} } \Bigr)^{\frac{2}
{3}}} + C\int_{{\R^3}} {{{| {\nabla \eta } |}^2}v_n^2( {{v_n}} )_L^4}  \hfill \\
   \le C {{{\Bigl( {\int_{{\R^3}} {{{( {\eta {v_n}( {{v_n}} )_L^2} )}^6}} } \Bigr)}^{\frac{1}
{3}}}\| {{v_n}} \|_{{L^6}\{ {| z | \ge {R \mathord{\left/
 {\vphantom {R 2}} \right.
 \kern-\nulldelimiterspace} 2}} \}}^4 + C\int_{{\R^3}}
 {{{| {\nabla \eta } |}^2}v_n^2( {{v_n}} )_L^4} } . \hfill \\
\end{gathered} \]
Since $v_n^6$ is uniformly integrable near infinity, $\exists \,\bar
R
> 1$, such that for any $R > \bar R$,
\[
\| {{v_n}} \|_{{L^6}\{ {| z | \ge {R \mathord{\left/
 {\vphantom {R 2}} \right.
 \kern-\nulldelimiterspace} 2}} \}}^4 \le \frac{1}
{{2C}}.
\]
Hence we get
\[
{\Bigl( {\int_{| z | \ge R} {{{( {{v_n}( {{v_n}} ) _L^2} )}^6}} }
\Bigr)^{\frac{1} {3}}} \le {\Bigl( {\int_{{\R^3}} {{{( {\eta {v_n} (
{{v_n}} )_L^2} )}^6}} } \Bigr)^{\frac{1} {3}}} \le C\int_{{\R^3}}
{{{| {\nabla \eta } |}^2} v_n^2( {{v_n}} )_L^4}  \le \frac{C}
{{{r^2}}}\int_{{\R^3}} {v_n^6} .
\]
Taking $r = \frac{R}
{2}$, we have
\[
{\Bigl( {\int_{| z | \ge R} {{{( {{v_n}( {{v_n}} )_L^2} )}^6}} }
\Bigr)^{\frac{1} {3}}} \le C\int_{{\R^3}} {v_n^6} .
\]
Letting $L \to \infty $, we get that
\[
\int_{| z | \ge R} {v_n^{18}}  \le C,
\]
which gives \eqref{2.40}.

Let  $t = \frac{9} {2}$, suppose ${v_n} \in {L^{2\beta t/(t - 1)}}\{
|z| \ge R - r\} $ for some $\beta  \ge 1$, \eqref{2.39},
\eqref{2.40} give that
\begin{eqnarray*}
  \| {{W_L}} \|_{{L^6}}^2 &\le& C{\beta ^2}{\Bigl( {\int_{|z| \ge R - r} {{{({\eta ^2}v_n^{2\beta })}^{t/(t - 1)}}} }
  \Bigr)^{1 - 1/t}}{\Bigl( {\int_{|z| \ge R - r} {(v_n^{18})} } \Bigr)^{1 - 1/t}}  \\
  &&+ C{\beta ^2}\frac{{{{({\R^3} - {{(R - r)}^3})}^{1/t}}}}
{{{r^2}}}
{\Bigl( {\int_{|z| \ge R - r} {(v_n^{2\beta t/(t - 1)})} } \Bigr)^{1 - 1/t}} \\
   &\le& C{\beta ^2}\Bigl( {1 + \frac{{{R^{3/t}}}}
{{{r^2}}}} \Bigr){\Bigl( {\int_{|z| \ge R - r} {(v_n^{2\beta t/(t - 1)})} } \Bigr)^{1 - 1/t}}.  \\
\end{eqnarray*}
Letting $L \to \infty $, we obtain
\[
\| {{v_n}} \|_{{L^{6\beta }}\{ |z| \ge R\} }^{2\beta }
\le C{\beta ^2}\Bigl( {1 + \frac{{{R^{3/t}}}}
{{{r^2}}}} \Bigr)
\| {{v_n}} \|_{{L^{2\beta t/(t - 1)}}\{ |z| \ge R\} }^{2\beta }.
\]
If we set $\chi  = 3(t - 1)/t$, $s = 2t/(t - 1)$, then
\[
{\| {{v_n}} \|_{{L^{\beta \chi s}}\{ |z| \ge R\} }}
\le {C^{1/\beta }}{\beta ^{1/\beta }}{(1 + \frac{{{R^{3/t}}}}
{{{r^2}}})^{1/2\beta }}
{\| {{v_n}} \|_{{L^{\beta s}}\{ |z| \ge R - r\} }}.
\]

Let $\beta  = {\chi ^m}$, $m = 1,2,...$, then we get
\begin{equation}\label{2.41}
{\| {{v_n}} \|_{{L^{{\chi ^{m + 1}}s}}\{ |z| \ge R\} }} \le {C^{{\chi ^{ - m}}}}{\chi ^m}^{{\chi ^{ - m}}}{\Bigl( {1 + \frac{{{R^{3/t}}}}
{{{r^2}}}} \Bigr)^{1/(2{\chi ^m})}}{\| {{v_n}} \|_{{L^{{\chi ^m}s}}\{ |z| \ge R - r\} }}.
\end{equation}
 If ${r_m} = {2^{ - ( {m + 1} )}}R$, then \eqref{2.41} implies
\begin{eqnarray*}
 \quad {\| {{v_n}} \|_{{L^{{\chi ^{m + 1}}s}}\{ |z| \ge R\} }} &\le& {\| {{v_n}} \|_{{L^{{\chi ^{m + 1}}s}}\{ |z| \ge R - {r_{m + 1}}\} }} \\
   &\le& {C^{\sum\nolimits_{i = 1}^m {{\chi ^{ - i}}} }}{\chi ^{\sum\nolimits_{i = 1}^m {i{\chi ^{ - i}}} }}\exp
   \Bigl( {\sum\limits_{i = 1}^m {\ln ( {2^{2(i + 1)}})/( {2{\chi ^i}} )} } \Bigr){\| {{v_n}} \|_{{L^{\chi s}}\{ |z| \ge R - {r_1}\} }}  \\
   &\le& C{\| {{v_n}} \|_{{L^6}\{ |z| \ge R/2\} }}.
\end{eqnarray*}
Letting $m \to \infty $, we get
\[
{\| {{v_n}} \|_{{L^\infty }\{ |z| \ge R\} }} \le C{\| {{v_n}} \|_{{L^6}\{ |z| \ge R/2\} }}.
\]

Since $\{ {v_n^6} \}$ is uniformly integrable near infinity,
\eqref{2.37} follows.
\end{proof}

\section{Proof of Theorem~\ref{1.1.}}

For $\varepsilon  > 0$, let ${v_\varepsilon }$ be the mountain-pass
solution to $({{{\hat E'}_\varepsilon }})$ given by Proposition
\ref{2.9.}. For any sequence $\{ {{\varepsilon _n}} \}$ satisfying
${{\varepsilon _n} \to 0^+}$, denote by ${v_n}: = {v_{{\varepsilon
_n}}}$, $J_{n}:=J_{\varepsilon_n}$ and $H_n:=H_{\varepsilon_n}$.
Then ${v_n}$ satisfies
\begin{equation}\label{3.1}
 - \Bigl( {a + b\int_{{\R^3}} {{{| {\nabla {v_n}} |}^2}} } \Bigr)
 \Delta {v_n} + V( {{\varepsilon _n}z} ){v_n}
 = g( {{\varepsilon _n}z,{v_n}} ){\text{  in  }}{\R^3}.
\end{equation}
Hence ${v_n}$ is a critical point of the following functional $J_n$,
and by \eqref{2.34}, ${{v_n}}$ is bounded in $H_n$.

Similar to Lemma \ref{2.7.}, we have

\begin{lemma}\label{3.1.}
There is a sequence $\{ {{y_n}} \} \subset {\R^3}$ and $R > 0$,
$\beta > 0$ such that
\[
\int_{{B_R}( {{y_n}} )} {v_n^2}  \ge \beta .
\]
\end{lemma}

\begin{lemma}\label{3.2.}
${\varepsilon _n}{y_n}$ is bounded in ${\R^3}$. Moreover,
${\text{dist}}( {{\varepsilon _n}{y_n},\Lambda '} ) \le {\varepsilon
_n}R$.
\end{lemma}

\begin{proof}
For $\delta  > 0$, define ${K_\delta } = \{ z \in
{\R^3}|{\text{dist}}(z,\Lambda ') \le \delta \} $. We set ${\phi
_{{\varepsilon _n}}}( z ) = \phi ( {{\varepsilon _n}z} )$ where
$\phi  \in {C^\infty }( {{\R^3},[ {0,1}]} )$ is such that
\[
\phi (z) = \left\{ \begin{gathered}
  1,{\text{  }}z \notin {K_\delta }, \hfill \\
  0,{\text{  }}z \in \Lambda',\hfill \\
\end{gathered}  \right.
\quad | {\nabla \phi } | \le \frac{C} {\delta }.
\]
 Taking ${v_n}{\phi _{{\varepsilon _n}}}$ as a test function
in \eqref{3.1}, using $(g_4)$ and the fact that ${\text{supp}}{\phi
_{{\varepsilon _n}}} \cap (\Lambda '/{\varepsilon _n}) = \emptyset
$, we get
\[\begin{gathered}
  \quad\alpha \Bigl( {1 - \frac{1}
{k}} \Bigr)\int_{{\R^3}} {v_n^2{\phi _{{\varepsilon _n}}}} \le
\Bigl( {1 - \frac{1} {k}} \Bigr)
\int_{{\R^3}} {V( {{\varepsilon _n}z} )v_n^2{\phi _{{\varepsilon _n}}}}  \hfill \\
 \le  - \Bigl( {a + b\int_{{\R^3}} {{{| {\nabla {v_n}} |}^2}} } \Bigr)\int_{{\R^3}} {{v_n}( {\nabla {\phi _{{\varepsilon _n}}} \cdot \nabla {v_n}} )}  \hfill \\
\le C\frac{{{\varepsilon _n}}} {\delta }\int_{{\R^3}} {| {{v_n}} ||
{\nabla {v_n}} |}  \le C\frac{{{\varepsilon _n}}}
{\delta }. \hfill \\
\end{gathered} \]
If there is a subsequence ${\varepsilon _{{n_j}}} \to {0^ + } $
such that
\[
{B_R}( {{y_{{n_j}}}} ) \cap \{ {z \in {\R^3},{\varepsilon _{{n_j}}}z
\in {K_\delta }} \} = \emptyset ,
\]
then
\[
\alpha \Bigl( {1 - \frac{1}
{k}} \Bigr)\int_{{B_R}( {{y_{{n_j}}}} )} {v_{{n_j}}^2}
\le C\frac{{{\varepsilon _{{n_j}}}}}
{\delta },
\]
which contradicts Lemma \ref{3.1.}. Thus, for all small
${{\varepsilon _n}}$ there is a ${{y'}_n}$ such that
 ${\varepsilon _n}{{y'}_n} \in {K_\delta }$ and
 $| {{{y'}_n} - {y_n}} | \le R$. It is easy to verify
 that ${\text{dist}}( {{\varepsilon _n}{y_n},\Lambda '} ) \le {\varepsilon _n}R + \delta $
 and by the arbitrariness of $\delta $, we complete the proof.
\end{proof}

From Lemma \ref{3.2.}, we can assume that
${\varepsilon _n}{y_n} \in \overline {\Lambda '} $
for all ${\varepsilon _n}$ small enough. Otherwise,
we can replace ${{y_n}}$ by $\varepsilon _n^{ - 1}{x_n}$
where ${x_n} \in \overline {\Lambda '} $ and $| {{y_n} - \varepsilon _n^{ - 1}{x_n}} | \le R$.
Thus
\[
0 < \beta  \le \int_{{B_R}( {{y_n}} )} {v_n^2}  \le \int_{{B_{2R}}( {\varepsilon _n^{ - 1}{x_n}} )} {v_n^2}
\]
and if we replace $R$ by $2R$ in
Lemma \ref{3.1.}, we have our claim.

\begin{lemma}\label{3.3.}
\begin{equation}\label{3.2}
\mathop {\lim }\limits_{n \to \infty } V( {{\varepsilon _n}{y_n}} ) = {V_0}.
\end{equation}
\end{lemma}

\begin{proof}
Since ${\varepsilon _n}{y_n} \in \overline {\Lambda '} $, up to
a subsequence, ${\varepsilon _n}{y_n} \to {x_0} \in \overline {\Lambda '} $,
we shall prove that $V( {{x_0}} ) = {V_0}$. We have
already known that $V( {{x_0}} ) \ge {V_0}$.
 Let we set ${w_n}(z) = {v_n}({z + {y_n}}) $, from \eqref{3.1} and
 Lemma \ref{3.1.}, we have
\[
\int_{{B_R}(0)} {w_n^2}  \ge \beta  > 0{\text{ for all }}n,
\]
\[
 - \Bigl( {a + b\int_{{\R^3}} {{{| {\nabla {w_n}} |}^2}} } \Bigr)
 \Delta {w_n} + V( {{\varepsilon _n}z + {\varepsilon _n}{y_n}} ){w_n}
 = g( {{\varepsilon _n}z + {\varepsilon _n}{y_n},{w_n}} )
 \]
and ${\| {{w_n}} \|_{{H^1}}} = {\| {{v_n}} \|_{{H^1}}}$ is bounded.
Up to a subsequence, $\exists w \in {H^1}( {{\R^3}} )\backslash \{ 0
\}$, such that
\begin{equation}\label{3.3}
\left\{ \begin{gathered}
  {w_n} \rightharpoonup w{\text{  in  }}{H^1}( {{\R^3}} ), \hfill \\
  {w_n} \to w{\text{  in  }}L_{{\text{loc}}}^p( {{\R^3}} ){\text{ for all }}1 \le p < 6, \hfill \\
  {w_n} \to w{\text{  a}}{\text{.e}}{\text{.}} \hfill \\
\end{gathered}  \right.
\end{equation}
and denote by $A: = \mathop {\lim }\limits_{n \to \infty }
\int_{{\R^3}} {{{| {\nabla {w_n}} |}^2}} $, it is clear that
$\int_{{\R^3}} {{{| {\nabla w} |}^2}}  \le A$.

Taking $\varphi  \in C_0^\infty ( {{\R^3}} )$ as a test function in
\eqref{3.1}, by \eqref{3.3},  we have
\[
  ( {a + bA})\int_{{\R^3}} {\nabla w\nabla \varphi }
  + \int_{{\R^3}} {V( {{x_0}} )w\varphi }  = \int_{{\R^3}} {\overline g (w)\varphi } {\text{    }}\forall \varphi  \in C_c^\infty ( {{\R^3}} ),
 \]
where $\overline g (w) = \chi ( {{x_0}} )( {f(w) + {w^5}} ) + ( {1 -
\chi ( {{x_0}} )} )\tilde f(w)$. By density, we get
\[
( {a + bA} )\int_{{\R^3}} {\nabla w\nabla \varphi } + \int_{{\R^3}}
{V( {{x_0}} )w\varphi } = \int_{{\R^3}} {\overline g (w)\varphi }
{\text{    }}\forall \varphi  \in {H^1}( {{\R^3}} ).
\]
Choose $\varphi  = w$, then
\[
\langle {{{\bar J'}_{{x_0}}}( w ),w} \rangle  \le 0,
\]
where, $\overline G \left( s \right) = \int_0^s {\overline g \left(
\tau  \right)} d\tau $ and
\[
{{\bar J}_{{x_0}}}(w) = \frac{1} {2}a\int_{{\R^3}} {{{| {\nabla w}
|}^2}}  + \frac{1} {2}\int_{{\R^3}} {V( {{x_0}} ){w^2}}  + \frac{1}
{4}b{\Bigl( {\int_{{\R^3}} {{{| {\nabla w} |}^2}} } \Bigr)^2}
 - \int_{{\R^3}} {\overline G (w)},\, {\text{  }}u \in
 {H^1}({{\R^3}}).
 \]
Moreover, with the same argument to prove \eqref{add}, we conclude
\[
\langle {{{\bar J'}_{{x_0}}}(w),w} \rangle  = 0
\]
and
\[
\int_{{\R^3}} {{{| {\nabla w} |}^2}}  = A:
= \mathop {\lim }\limits_{n \to \infty } \int_{{\R^3}} {{{| {\nabla
{w_n}} |}^2}} ,
\]
which implies that  $w>0$ is a critical point of ${{\bar
J}_{{x_0}}}$.

Now we prove $V({{x_0}})=V_0$. Assuming to the contrary that
$V({{x_0}})
> {V_0}$. Denote by
\[{c_{V( {{x_0}} )}}: =
\mathop {\inf }\limits_{u \in {H^1}({{\R^3}})\backslash \{0\}}
\mathop {\sup }\limits_{\tau  \ge 0} {I_{V( {{x_0}} )}} ( {\tau u}
).\]

Let $c_{x_0}$ be the mountain-pass energy of $\bar J_{x_0}$. Then
${c_{{x_0}}} \ge {c_{V( {{x_0}} )}}$ since ${{\bar J}_{{x_0}}}(u)
\ge {I_{V( {{x_0}} )}}( u )$. Hence
\begin{equation}\label{3.6}
\begin{gathered}
  \quad{c_{V( {{x_0}} )}}
 \le {c_{{x_0}}} \le {{\bar J}_{{x_0}}}( w ) - \frac{1}
{4}\langle {{{\bar J'}_{{x_0}}}( w ),w} \rangle  \hfill \\
   = \frac{1}
{4}\int_{{\R^3}} {a{{| {\nabla w} |}^2} + V( {{x_0}} ){w^2}} +
\int_{{\R^3}} {\Bigl( {\frac{1}
{4}\bar g(w)w - \bar G(w)} \Bigr)}  \hfill \\
   \le \mathop {\underline {\lim } }\limits_{n \to \infty } \frac{1}
{4}\int_{{\R^3}} {a{{| {\nabla {w_n}} |}^2} + V( {{\varepsilon _n}z
+ {\varepsilon _n}{y_n}} )w_n^2}   \hfill \\
  \quad+\int_{{\R^3}} {\Bigl( {\frac{1}
{4}g( {{\varepsilon _n}z + {\varepsilon _n}{y_n},{w_n}} ){w_n}
- G( {{\varepsilon _n}z + {\varepsilon _n}{y_n},{w_n}} )} \Bigr)}  \hfill \\
   = \mathop {\underline {\lim } }\limits_{n \to \infty } \frac{1}
{4}\int_{{\R^3}} {a{{| {\nabla {v_n}} |}^2} + V( {{\varepsilon _n}z}
)v_n^2} + \int_{{\R^3}} {\Bigl( {\frac{1}
{4}g( {{\varepsilon _n}z,{v_n}} ){v_n} - G( {{\varepsilon _n}z,{v_n}} )} \Bigr)}  \hfill \\
   = \mathop {\underline {\lim } }\limits_{n \to \infty } {J_n}( {{v_n}} ) - \frac{1}
{4}\langle {{{J'}_n}( {{v_n}} ),{v_n}} \rangle  \le {c_{{V_0}}}. \hfill \\
\end{gathered}
\end{equation}

Denote ${w'}$ be a critical point of ${I_{V( {{x_0}} )}}$ with minimal energy,
there exists a $t' > 0$ such that
\[
{I_{{V_0}}}( {t'w'} ) = \mathop {\sup }\limits_{t > 0} {I_{{V_0}}}( {tw'} ).
\]
Since $V( {{x_0}} ) > {V_0}$, we have
\[
\mathop {\sup }\limits_{t > 0} {I_{{V_0}}}( {tw'} )
= {I_{{V_0}}}( {t'w'} ) < {I_{V( {{x_0}} )}}( {t'w'} ) \le \mathop {\sup }\limits_{t > 0} {I_{V( {{x_0}} )}}( {tw'} )
= {I_{V( {{x_0}} )}}( {w'} ) = {c_{V( {{x_0}} )}},
\]
then ${c_{{V_0}}} < {c_{V( {{x_0}} )}}$, which contradicts \eqref{3.6}, thus \eqref{3.2} follows.
\end{proof}
\begin{proof}[\bf Proof of Theorem~\ref{1.1.}]

Since $V( {{x_0}} ) = {V_0}$, then ${c_{{V_0}}} = {c_{V( {{x_0}} )}}$.
Combining with \eqref{3.6}, we get
\[
\mathop {\lim }\limits_{n \to \infty } \int_{{\R^3}} {{{| {\nabla
{w_n}} |}^2}} = \int_{{\R^3}} {{{| {\nabla w} |}^2}}.
\]
From Sobolev's inequality, $\{ {{{| {{w_n}} |}^6}} \}$ is uniformly
integrable near infinity. Lemma \ref{2.12.} yields
\begin{equation}\label{3.7}
\mathop {\lim }\limits_{|z| \to \infty } {w_n}(z)
 = 0{\text{ uniformly for }}n.
\end{equation}
which implies that there is a $\rho $ such that ${w_n}(z) < a'$ for
all $|z| \ge \rho $ and large $n$, that is
\[
 - \Bigl( {a + b\int_{{\R^3}} {{{| {\nabla {w_n}} |}^2}} } \Bigr)\Delta {w_n}
 + V( {{\varepsilon _n}z + {\varepsilon _n}{y_n}} ){w_n}
  = f( {{w_n}} ) + w_n^5{\text{  in  }}|z| \ge \rho .
 \]
On the other hand, if $|z| \le \rho $, by Lemma \ref{3.3.},  we get
${B_{{\varepsilon _n}\rho }}( {{\varepsilon _n}{y_n}} ) \subset
\Lambda $
  for all ${{\varepsilon _n}}$ small enough.
  So $g( {{\varepsilon _n}z + {\varepsilon _n}{y_n},{w_n}} ) = f( {{w_n}} ) +
  w_n^5$ and
\begin{equation}\label{3.8}
 - \Bigl( {a + b\int_{{\R^3}} {{{| {\nabla {w_n}} |}^2}} } \Bigr)\Delta {w_n}
 + V( {{\varepsilon _n}z + {\varepsilon _n}{y_n}} ){w_n}
  = f( {{w_n}} ) + w_n^5{\text{  in  }}{\R^3}.
\end{equation}
 Combining with the arbitrariness of
 $\{ {{\varepsilon _n}} \}$, we have obtained the
 existence of solutions ${v_\varepsilon }$
 for $( {{{\hat E}_\varepsilon }} )$, which
 is equivalent to the existence of solutions
 ${u_\varepsilon }$ for problem $( {{E_\varepsilon }} )$.

 Now we claim that
if ${P_n }$ is a maximum of ${w_n }$, then \[{w_n }( {{P_n }} ) \ge
a'\] for all $n$.

Indeed, if ${w_n }( {{P_n }} ) < a'$, taking ${w_n }$ as a test
function for \eqref{3.8}, we get
\[
\int_{{\R^3}} {V( {{\varepsilon _n}z + {\varepsilon _n}{y_n}}
)w_n^2}  \le \int_{{\R^3}} {f( {{w_n}} ){w_n} + w_n^6} ,
\]
which gives
\begin{eqnarray*}
 \alpha \int_{{\R^3}} {w_n^2}  &\le& \int_{{\R^3}} {f( {{w_n}} ){w_n} + w_n^6}
  = \int_{{\R^3}} {w_n^2\Bigl( {\frac{{f( {{w_n}} )}}
{{{w_n}}} + w_n^4} \Bigr)}   \\
   &\le& \int_{{\R^3}} {w_n^2\Bigl( {\frac{{f\left( {a'} \right)}}
{{a'}} + {{( {a'} )}^4}} \Bigr)}  = \frac{\alpha }
{k}\int_{{\R^3}} {w_n^2},\\
\end{eqnarray*}
where $k > 2$. Hence we got a contradiction.

By \eqref{3.7}, ${{P_n }}$ must be bounded. Denote ${z_n} =
{\varepsilon _n}{P_n} + {\varepsilon _n}{y_n}$, it is clear that
${z_n }$ is a maximum of ${u_{{\varepsilon _n}}}$. Combining with
Lemma \ref{3.3.} and the arbitrariness of $\{ {{\varepsilon _n}}
\}$, we have obtained the concentration result in Theorem
\ref{1.1.}.

To complete the proof, we only need to  prove the exponential decay
of ${u_\varepsilon }$. Since the proof is standard (see
\cite{pf,wtxz}, for example), we omit it here.
\end{proof}

\section{Multiplicity of solutions to $({E_\varepsilon })$}

Suppose that $V$ be a Banach space, $\mathcal{V}$ be a ${C^1}$-manifold
of $V$ and $I:V \to R$ a ${C^1}$-functional. We say that $I| {_\mathcal{V}} $
satisfies the $(P.S.)$ condition at level $c$ (${( {P.S.} )_c}$ in short)
if any sequence $\{ {{u_n}} \} \subset \mathcal{V}$ such that $I( {{u_n}} ) \to c$
and ${\| {I'( {{u_n}} )} \|_ * } \to 0$ contains a convergent subsequence.
Here ${\| {I'(u)} \|_ * }$ denotes the norm of the derivative
of $I$ restricted to $\mathcal{V}$ at the point $u \in \mathcal{V}$.

\begin{proposition}\label{3.5.} The functional restricted
to ${\mathcal{N}_\varepsilon }$ satisfies ${( {P.S.} )_c}$
condition for each $c \in \Bigl( {0,\frac{1}
{4}ab{S^3} + \frac{1}
{{24}}{b^3}{S^6} + \frac{1}
{{24}}{{( {{b^2}{S^4} + 4aS} )}^{\frac{3}
{2}}}} \Bigr)$, where
\[
{\mathcal{N}_\varepsilon } : = \{ u \in {H_\varepsilon }\backslash
\{ 0\} |\langle {{{J'}_\varepsilon }( u ),u} \rangle  = 0\}.
\]

\end{proposition}

\begin{proof}
Let $\{ {{u_n}} \} \subset {\mathcal{N}_\varepsilon }$ be such that
\begin{equation}\label{3.12}
{J_\varepsilon }( {{u_n}} ) \to c{\text{ and }}{\| {{{J'}_\varepsilon }
( {{u_n}} )} \|_ * } \to 0{\text{ as }}n \to \infty .
\end{equation}
There exists $\{ {{\lambda _n}} \} \subset R$ such that
\[
{{J'}_\varepsilon }( {{u_n}} ) = {\lambda _n}{{\phi '}_\varepsilon }( {{u_n}} ) + o(1),
\]
where
\[
{\phi _\varepsilon }(u) = \langle {{{J'}_\varepsilon }(u),u} \rangle .
\]
Since $\{ {{u_n}} \} \subset {\mathcal{N}_\varepsilon }$, we have that
\[
0 = \langle {{{J'}_\varepsilon }( {{u_n}} ),{u_n}} \rangle
= {\lambda _n}\langle {{{\phi '}_\varepsilon }( {{u_n}} ),{u_n}} \rangle
 + o(1){\| {{u_n}} \|_\varepsilon }.
\]
Direct calculations show that $\{ {{u_n}} \}$ is bounded in ${{H_\varepsilon }}$, we have that
\begin{equation}\label{3.13}
\left\{ \begin{gathered}
  {u_n} \rightharpoonup u{\text{ in }}{H_\varepsilon }, \hfill \\
  {u_n} \to u{\text{ in }}L_{{\text{loc}}}^s{\text{ 1}} \le s < 6, \hfill \\
  {u_n} \to u{\text{ a}}{\text{.e}}{\text{.}} \hfill \\
\end{gathered}  \right.
\end{equation}
\eqref{3.13} and the fact $( {\frac{{f(s)}} {{{s^3}}}} )^\prime>0
,{( {\frac{{\tilde f(s)}} {s}} )^\prime } \ge 0$ for all $s \ge 0$
imply that
\[\begin{gathered}
  \quad\langle {{{\phi '}_\varepsilon }( {{u_n}} ),{u_n}} \rangle  \hfill \\
   = 2\int_{{\R^3}} (a {{{| {\nabla {u_n}} |}^2}}
   + {V( {\varepsilon z} )u_n^2})
   + 4b{\Bigl( {\int_{{\R^3}} {{{| {\nabla {u_n}} |}^2}} } \Bigr)^2}
   - \int_{{\R^3}} {( {g'( {\varepsilon z,{u_n}} )u_n^2
   + g( {\varepsilon z,{u_n}} ){u_n}} )}  \hfill \\
   =  - 2a\int_{{\R^3}} {{{| {\nabla {u_n}} |}^2}}  - 2\int_{{\R^3}} {V( {\varepsilon z} )u_n^2}
   + \int_{{\R^3}} {( {3g( {\varepsilon z,{u_n}} ){u_n} - g'( {\varepsilon z,{u_n}} )u_n^2} )}  \hfill \\
   =  - 2a\int_{{\R^3}} {{{| {\nabla {u_n}} |}^2}}
   - 2\int_{{\R^3}} {V( {\varepsilon z} )u_n^2}  + \int_{{\R^3}} {\chi ( {\varepsilon z} )\Bigl( {3f( {{u_n}} ){u_n}
   - f'( {{u_n}} )u_n^2 - 2{{( {u_n^ + } )}^6}} \Bigr)}  \hfill \\
   \quad+ \int_{{\R^3}} {( {1 - \chi ( {\varepsilon z} )} )\Bigl( {3\tilde f( {{u_n}} ){u_n} - \tilde f'( {{u_n}} )u_n^2} \Bigr)}  \hfill \\
   \le  - 2a\int_{{\R^3}} {{{| {\nabla {u_n}} |}^2}}
   - 2\int_{{\R^3}} {V( {\varepsilon z} )u_n^2}
   + \int_{{\R^3}} {2( {1 - \chi ( {\varepsilon z} )} )\tilde f( {{u_n}} ){u_n}}  \hfill \\
   \le  - 2a\int_{{\R^3}} {{{| {\nabla {u_n}} |}^2}}
   - 2\int_{{\R^3}} {V( {\varepsilon z} )u_n^2}
   + \frac{2}
{k}\int_{{\R^3}} {V( {\varepsilon z} )u_n^2}  \hfill \\
   \le  - 2\Bigl( {1 - \frac{1}
{k}} \Bigr)\Bigl[ {a\int_{{\R^3}} {{{| {\nabla {u_n}} |}^2}}
+ \int_{{\R^3}} {V( {\varepsilon z} )u_n^2} } \Bigr] \hfill \\
   =  - 2\Bigl( {1 - \frac{1}
{k}} \Bigr)\| {{u_n}} \|_\varepsilon ^2. \hfill \\
\end{gathered} \]
We may suppose that $\langle {{{\phi '}_\varepsilon }( {{u_n}}
),{u_n}} \rangle  \to l < 0$. Hence the above expression shows that
${\lambda _n} \to 0$ and therefore we conclude
 that ${{J'}_\varepsilon }( {{u_n}} ) \to 0$ as $n \to \infty $ in the
 dual space of ${{H_\varepsilon }}$. Now, we claim that,
 for each $\delta  > 0$, there exists $R > 0$ such that
\begin{equation}\label{3.14}
\mathop {\overline {\lim } }\limits_{n \to \infty }
\int_{{\R^3}\backslash {B_R}(0)} {( {a{{| {\nabla {u_n}} |}^2} + V(
{\varepsilon z} )u_n^2} )}  < \delta .
\end{equation}
In fact, first, we may assume that $R$ is chosen so that $(\Lambda
'/\varepsilon ) \subset {B_{R/2}}(0)$. Let ${\eta _R}$ be a smooth
cut-off function such that ${\eta _R} = 0$ on ${B_{R/2}}(0)$, ${\eta
_R} = 1$ on ${{\R^3}\backslash {B_R}(0)}$, $0 \le {\eta _R} \le 1$
and $| {\nabla {\eta _R}} | \le \frac{C} {R}$. Since $\{ {{u_n}} \}$
is a bounded $(P.S.)$ sequence, we have
\[
\langle {{{J'}_\varepsilon }({{u_n}}),{\eta _R}{u_n}} \rangle  \to 0{\text{ as }}n \to \infty .
\]
Thus
\[\begin{gathered}
  \quad a\int_{{\R^3}} {( {\nabla {u_n}} ) \cdot ( {\nabla ( {{\eta _R}{u_n}} )} )}
  + \int_{{\R^3}} {V( {\varepsilon z} )u_n^2{\eta _R}}
  + b\int_{{\R^3}} {{{| {\nabla {u_n}} |}^2}} \int_{{\R^3}} {( {\nabla {u_n}} ) \cdot ( {\nabla ( {{\eta _R}{u_n}} )} )}  \hfill \\
   = \int_{{\R^3}} {g( {\varepsilon z,{u_n}} ){u_n}{\eta _R}}  + o(1) \le \frac{1}
{k}\int_{{\R^3}} {V( {\varepsilon z} )u_n^2{\eta _R}}  + o(1). \hfill \\
\end{gathered} \]
We conclude that
\[\begin{gathered}
  \quad a\int_{{\R^3}\backslash {B_R}( 0 )} {{{| {\nabla {u_n}} |}^2}}  + \Bigl( {1 - \frac{1}
{k}} \Bigr)\int_{{\R^3}\backslash {B_R}(0)} {V({\varepsilon z})u_n^2}  \hfill \\
   \le \frac{C}
{R}{\| {\nabla {u_n}} \|_{{L^2}( {{\R^3}} )}}{\| {{u_n}} \|_{{L^2}(
{{\R^3}} )}} + \frac{C}
{R}\| {\nabla {u_n}} \|_{{L^2}( {{\R^3}} )}^3{\| {{u_n}} \|_{{L^2}( {{\R^3}} )}} + o( 1 ), \hfill \\
\end{gathered} \]
and \eqref{3.14} follows.

We claim that
\begin{equation}\label{3.15}
\int_{{\R^3}} {g( {\varepsilon z,{u_n}} ){u_n}}  \to \int_{{\R^3}}
{g( {\varepsilon z,u} )u} .
\end{equation}
In fact, we can use \eqref{3.13} and Dominated
Convergence Theorem to show that
\[
\int_{{B_R}( 0 )} {\chi ( {\varepsilon z} )f( {{u_n}} ){u_n}}
\to \int_{{B_R}( 0 )} {\chi ( {\varepsilon z} )f(u)u}
\]
and
\[
\int_{{B_R}( 0 )} {( {1 - \chi ( {\varepsilon z} )} )\tilde f( {{u_n}} ){u_n}}
\to \int_{{B_R}( 0 )} {( {1 - \chi ( {\varepsilon z} )} )\tilde f( u )u} .
\]
In order to get \eqref{3.15}, we just need to prove that
\begin{equation}\label{3.16}
\int_{{B_R}( 0 )} {\chi ( {\varepsilon z} ){{( {u_n^ + } )}^6}}  \to
\int_{{B_R}( 0 )} {\chi ( {\varepsilon z} ){{( {{u^ + }} )}^6}} .
\end{equation}
Since $\{ {{u_n}} \}$ is bounded in ${H^1}( {{\R^3}} )$, we may
suppose that
\[
{| {\nabla u_n^ + } |^2} \rightharpoonup {| {\nabla {u^ + }} |^2} +
\mu {\text{ and }}{| {u_n^ + } |^6} \rightharpoonup {| {{u^ + }}
|^6} + \nu,
\]
where $\mu$ and $\nu$ are bounded nonnegative measure in $\R^3$.  By
the Concentration Compactness Principle II (Lemma 1.1 of \cite{l3}),
we obtain an at most countable index set $\Gamma $, sequence $\{
{{x_i}} \} \subset {\R^3}$ and $\{ {{\mu _i}} \},\{ {{\nu _i}} \}
\subset ( {0,\infty } )$ such that
\begin{equation}\label{3.17}
\mu  \ge \sum\limits_{i \in \Gamma } {{\mu _i}} {\delta
_{{x_i}}},\nu = \sum\limits_{i \in \Gamma } {{\nu _i}} {\delta
_{{x_i}}}{\text{ and }}S{( {{\nu _i}} )^{\frac{1} {3}}} \le {\mu
_i}.
\end{equation}

It suffices to show that ${\{ {{x_i}} \}_{i \in \Gamma }} \cap \{ {z| {\chi ( {\varepsilon z} ) > 0} } \} = \emptyset $. Suppose, by contradiction,
that $\chi ( {\varepsilon {x_i}} ) > 0$ for some $i \in \Gamma $. Define,
for $\rho  > 0$, the function ${\psi _\rho }( z ): = \psi ( {\frac{{z - {x_i}}}
{\rho }} )$ where $\psi$ is a smooth cut-off
 function such that $\psi  = 1$ on ${B_1}( 0 )$,
 $\psi  = 0$ on ${\R^3}\backslash {B_2}( 0 )$, $0 \le \psi  \le 1$
 and $| {\nabla \psi } | \le C$. we suppose that $\rho $ is chosen
 in such a way that the support of ${\psi _\rho }$ is
 contained in $\{ {z| {\chi ( {\varepsilon z} ) > 0} } \}$.
We see
\[
\langle {{{J'}_\varepsilon }( {{u_n}} ),{\psi _\rho }u_n^ + } \rangle
\to 0{\text{ as }}n \to \infty ,
\]
i.e.
\begin{equation}\label{3.18}
\begin{gathered}
  a\int_{{\R^3}} {{{| {\nabla u_n^ + } |}^2}{\psi _\rho }}
  + a\int_{{\R^3}} {( {\nabla u_n^ +  \cdot \nabla {\psi _\rho }} )u_n^ + }
  + \int_{{\R^3}} {V( {\varepsilon z} ){{( {u_n^ + } )}^2}{\psi _\rho }}  \hfill \\
   + b\Bigl( {\int_{{\R^3}} {{{| {\nabla {u_n}} |}^2}} } \Bigr)\Bigl( {\int_{{\R^3}} {{{| {\nabla u_n^ + } |}^2}{\psi _\rho }} } \Bigr)
   + b\Bigl( {\int_{{\R^3}} {{{| {\nabla {u_n}} |}^2}} } \Bigr)\Bigl( {\int_{{\R^3}} {( {\nabla u_n^ +  \cdot \nabla {\psi _\rho }} )u_n^ + } } \Bigr) \hfill \\
   - \int_{{\R^3}} {g( {\varepsilon z,{u_n}} )u_n^
   + {\psi _\rho }}  = o( 1 ). \hfill \\
\end{gathered}
\end{equation}
Since
\[\begin{gathered}
\quad  \overline {\mathop {\lim }\limits_{n \to \infty } } \Bigl|
{\int_{{\R^3}} {( {\nabla u_n^
  +  \cdot \nabla {\psi _\rho }} )u_n^ + } } \Bigr| \le \overline {\mathop {\lim }
  \limits_{n \to \infty } } {\Bigl( {\int_{{\R^3}} {{{| {\nabla {u_n}} |}^2}} } \Bigr)^{\frac{1}
{2}}} \cdot {\Bigl( {\int_{{\R^3}} {u_n^2{{| {\nabla {\psi _\rho }}
|}^2}} } \Bigr)^{\frac{1}
{2}}} \hfill \\
   \le  C{\Bigl( {\int_{{\R^3}} {{u^2}{{| {\nabla {\psi _\rho }}
|}^2}} } \Bigr)^{\frac{1} {2}}} \le C{\Bigl( {\int_{{B_{2\rho }}(
{{x_i}} )} {{u^6}} } \Bigr)^{\frac{1} {6}}}{\Bigl( {\int_{{B_{2\rho
}}( {{x_i}} )} {{{| {\nabla {\psi _\rho }} |}^3}} } \Bigr)^{\frac{1}
{3}}} \hfill \\
   \le C{\Bigl( {\int_{{B_{2\rho }}( {{x_i}} )} {{u^6}} } \Bigr)^{\frac{1}
{6}}} \to 0{\text{ as }}\rho  \to 0, \hfill \\
\end{gathered} \]

\[
\overline {\mathop {\lim }\limits_{n \to \infty } } a\int_{{\R^3}}
{{{| {\nabla u_n^ + } |}^2}{\psi _\rho }}  \ge a\int_{{\R^3}} {{{|
{\nabla {u^ + }} |}^2}{\psi _\rho }} + a{\mu _i} \to a{\mu
_i}{\text{ as }}\rho  \to 0,
\]

\[\begin{gathered}
  \quad\overline {\mathop {\lim }\limits_{n \to \infty } } b\Bigl( {\int_{{\R^3}} {{{| {\nabla {u_n}} |}^2}} } \Bigr)\Bigl( {\int_{{\R^3}} {{{| {\nabla u_n^ + } |}^2}{\psi _\rho }} } \Bigr) \ge \overline {\mathop {\lim }\limits_{n \to \infty } } b{\Bigl( {\int_{{\R^3}} {{{| {\nabla u_n^ + } |}^2}{\psi _\rho }} } \Bigr)^2} \hfill \\
   \ge b{\Bigl( {\int_{{\R^3}} {{{| {\nabla {u^ + }} |}^2}{\psi _\rho }}  + {\mu _i}} \Bigr)^2} \to b\mu _i^2{\text{ as }}\rho  \to 0, \hfill \\
\end{gathered} \]

\[
\overline {\mathop {\lim }\limits_{n \to \infty } } \int_{{\R^3}}
{V( {\varepsilon z} ){{( {u_n^ + } )}^2}{\psi _\rho }}  =
\int_{{\R^3}} {V( {\varepsilon z} ){{( {{u^ + }} )}^2}{\psi _\rho }}
\to 0{\text{ as }}\rho  \to 0,
\]
and similarly,
\[\begin{gathered}
  \quad \overline {\mathop {\lim }\limits_{n \to \infty } } \int_{{\R^3}} {g( {\varepsilon z,{u_n}} )u_n^ + {\psi _\rho }}  \hfill \\
   = \int_{{\R^3}} {\chi ( {\varepsilon z} )f( u )u{\psi _\rho }}
   + \int_{{\R^3}} {( {1 - \chi ( {\varepsilon z} )} )\tilde f( u )u{\psi _\rho }}
   + \int_{{\R^3}} {\chi ( {\varepsilon z} ){{( {{u^ + }} )}^6}{\psi _\rho }}
   + \chi ( {\varepsilon {x_i}} ){\nu _i} \hfill \\
   \to \chi ( {\varepsilon {x_i}} ){\nu _i}{\text{ as }}\rho  \to 0, \hfill \\
\end{gathered} \]
we  obtain from \eqref{3.18} that
\[
a{\mu _i} + b\mu _i^2 \le \chi ( {\varepsilon {x_i}} ){\nu _i}.
\]
Combining with \eqref{3.17}, we have
\[
{( {{\nu _i}} )^{\frac{1}
{3}}} \ge \frac{{b{S^2} + \sqrt {{b^2}{S^4} + 4aS} }}
{{2\chi ( {\varepsilon {x_i}} )}}.
\]
On the other hand,
\[\begin{gathered}
  c + o( 1 ) \hfill \\
   = {J_\varepsilon }( {{u_n}} ) - \frac{1}
{4}\langle {{{J'}_\varepsilon }( {{u_n}} ),{u_n}} \rangle  \hfill \\
   \ge \frac{1}
{4}a\int_{{\R^3}} {{{| {\nabla u_n^ + } |}^2}} + \frac{1}
{4}\int_{{\R^3}} {V( {\varepsilon z} ){{( {u_n^ + } )}^2}} +
\int_{{\R^3}} {\Bigl( {\frac{1}
{4}g( {\varepsilon z,{u_n}} ){u_n} - G( {\varepsilon z,{u_n}} )} \Bigr)}  \hfill \\
   = \frac{1}
{4}a\int_{{\R^3}} {{{| {\nabla u_n^ + } |}^2}} + \frac{1}
{4}\int_{{\R^3}} {V( {\varepsilon z} ){{( {u_n^ + } )}^2}} +
\int_{{\R^3}} {\chi ( {\varepsilon z} )\Bigl( {\frac{1}
{4}f( {{u_n}} ){u_n} - F( {{u_n}} )} \Bigr)}  \hfill \\
   \quad+ \frac{1}
{{12}}\int_{{\R^3}} {\chi ( {\varepsilon z} ){{( {u_n^ + } )}^6}}
 + \int_{{\R^3}} {( {1 - \chi ( {\varepsilon z} )} )\Bigl( {\frac{1}
{4}\tilde f( {{u_n}} ){u_n} - \tilde F( {{u_n}} )} \Bigr)}  \hfill \\
   \ge \frac{1}
{4}a\int_{{\R^3}} {{{| {\nabla u_n^ + } |}^2}}  + \frac{1}
{4}\int_{{\R^3}} {V( {\varepsilon z} ){{( {u_n^ + } )}^2}} +
\frac{1} {{12}}\int_{{\R^3}} {\chi ( {\varepsilon z} ){{( {u_n^ + }
)}^6}} - \frac{1}
{2}\int_{{\R^3}} {( {1 - \chi ( {\varepsilon z} )} )\tilde F( {{u_n}} )}  \hfill \\
   \ge \frac{1}
{4}a\int_{{\R^3}} {{{| {\nabla u_n^ + } |}^2}}  + \frac{1}
{4}\int_{{\R^3}} {V( {\varepsilon z} ){{( {u_n^ + } )}^2}} +
\frac{1}
{{12}}\int_{{\R^3}} {\chi ( {\varepsilon z} ){{( {u_n^ + } )}^6}}  \hfill \\
  {\text{   }} - \frac{1}
{{4k}}\int_{{\R^3}} {( {1 - \chi ( {\varepsilon z} )} )V( {\varepsilon z} ){{( {u_n^ + } )}^2}}  \hfill \\
   \ge \frac{1}
{4}a\int_{{\R^3}} {{{| {\nabla u_n^ + } |}^2}} + \frac{1}
{{12}}\int_{{\R^3}} {\chi ( {\varepsilon z} ){{( {u_n^ + } )}^6}}  \hfill \\
   \ge \frac{1}
{4}a{\mu _i} + \frac{1}
{{12}}\chi ( {\varepsilon {x_i}} ){\nu _i} + o( 1 ), \hfill \\
\end{gathered} \]
and hence
\begin{eqnarray*}
  c &\ge&
    \frac{1}
{4}aS{( {{\nu _i}} )^{\frac{1}
{3}}} + \frac{1}
{{12}}\chi ( {\varepsilon {x_i}} ){\nu _i}  \\
   &\ge& \frac{1}
{4}aS\frac{{b{S^2} + \sqrt {{b^2}{S^4} + 4aS} }}
{{2\chi ( {\varepsilon {x_i}} )}} + \frac{1}
{{12}}\chi ( {\varepsilon {x_i}} ){\Bigl( {\frac{{b{S^2}
+ \sqrt {{b^2}{S^4} + 4aS} }}
{{2\chi ( {\varepsilon {x_i}} )}}} \Bigr)^3}  \\
   &\ge& \frac{1}
{4}aS\frac{{b{S^2} + \sqrt {{b^2}{S^4} + 4aS} }}
{2} + \frac{1}
{{12}}{\Bigl( {\frac{{b{S^2} + \sqrt {{b^2}{S^4} + 4aS} }}
{2}} \Bigr)^3}  \\
   &=& \frac{1}
{4}ab{S^3} + \frac{1}
{{24}}{b^3}{S^6} + \frac{1}
{{24}}{\left( {{b^2}{S^4} + 4aS} \right)^{\frac{3}
{2}}}. \\
\end{eqnarray*}
This leads to a contradiction, hence \eqref{3.16} holds, then \eqref{3.15} follows.

Since $\langle {{{J'}_\varepsilon }( {{u_n}} ),{u_n}} \rangle  \to 0$, then
\begin{equation}\label{3.19}
\| {{u_n}} \|_\varepsilon ^2 + b{\Bigl( {\int_{{\R^3}} {{{| {\nabla
{u_n}} |}^2}} } \Bigr)^2} - \int_{{\R^3}} {g( {\varepsilon z,{u_n}}
){u_n}}  \to 0{\text{ as }}n \to \infty .
\end{equation}
By \eqref{3.13}, $u$ satisfies
\[ - (a + bA)\Delta u + V( {\varepsilon z} )u = g( {\varepsilon z,u} ),\]
where $A: = \mathop {\lim }\limits_{n \to \infty } \int_{{\R^3}}
{{{| {\nabla {u_n}} |}^2}}  \ge \int_{{\R^3}} {{{| {\nabla u} |}^2}}
$. Then
\begin{equation}\label{3.20}
\| u \|_\varepsilon ^2 + bA\int_{{\R^3}} {{{| {\nabla {u_n}} |}^2}}
- \int_{{\R^3}} {g( {\varepsilon z,u} )u}  = 0.
\end{equation}
Combining \eqref{3.15}, \eqref{3.19} with \eqref{3.20}, we get
\[
{u_n} \to u{\text{ in }}{H_\varepsilon }{\text{ as }}n \to \infty .
\]
\end{proof}

\begin{proposition}\label{3.6.}
For any $\delta  > 0$, there exists ${\varepsilon _\delta } > 0$
such that for any $\varepsilon  \in ( {0,{\varepsilon _\delta }} )$,
the Equation $( {{{\hat E'}_\varepsilon }} )$ has
 at least ${\text{ca}}{{\text{t}}_{{M_\delta }}}( M )$ solutions.
\end{proposition}

Before proving this proposition, we need some lemmas.\\

\begin{lemma}\label{3.7.}
(See Chapter II, 3.2. of \cite{c}) Let $I$ be a ${C^1}$-functional
defined on a ${C^1}$-Finsler manifold $\mathcal{V}$. If $I$
is bounded from below and satisfies the (PS) condition,
then $I$ has at least
${\text{ca}}{{\text{t}}_\mathcal{V}}( \mathcal{V} )$
 distinct critical points.\\
\end{lemma}

\begin{lemma}\label{3.8.}
(See Lemma 4.3 of \cite{bc})  Let $\Gamma $, ${\Omega ^ + }$, ${\Omega ^ - }$
be closed sets with ${\Omega ^ - } \subset {\Omega ^ + }$.
 Let $\Phi :{\Omega ^ - } \to \Gamma $, $\beta :\Gamma  \to {\Omega ^ + }$
 be two continuous maps such that $\beta  \circ \Phi $
 is homotopically equivalent to the embedding
  $Id:{\Omega ^ - } \to {\Omega ^ + }$.
  Then
 ${\text{ca}}{{\text{t}}_\Gamma }( \Gamma  )
 \ge {\text{ca}}{{\text{t}}_{{\Omega ^ + }}}( {{\Omega ^ - }} )$.\\
\end{lemma}

From Proposition \ref{2.10.}, denote by $w \in {H^1}( {{\R^3}} )$
such that ${{I'}_{{V_0}}}( w ) = 0$ and ${I_{{V_0}}}( w ) =
{c_{{V_0}}}$, where ${I_{{V_0}}}$, ${c_{{V_0}}}$ have been mentioned
in \eqref{2.23}.

Let us consider $\delta  > 0$ such that ${M_\delta } \subset \Lambda
$ and a smooth cut-off function $\eta $ with $0 \le \eta  \le 1$,
$\eta  = 1$ on ${B_1}( 0 )$, $\eta  = 0$ on ${\R^3}\backslash {B_2}(
0 )$, $| {\nabla \eta } | \le C$. For any $y \in M$, we define the
function
\[
{\psi _{\varepsilon ,y}}( z ) = \eta \Bigl( {\frac{{\varepsilon z - y}}
{{\sqrt \varepsilon  }}} \Bigr)w\Bigl( {\frac{{\varepsilon z - y}}
{\varepsilon }} \Bigr)
\]
and ${t_\varepsilon } > 0$ satisfying
$\mathop {\max }\limits_{t \ge 0} {J_\varepsilon }( {t{\psi _{\varepsilon ,y}}} )
 = {J_\varepsilon }( {{t_\varepsilon }{\psi _{\varepsilon ,y}}} )$
  and
$\frac{{d{J_\varepsilon }( {t{\psi _{\varepsilon ,y}}} )}}
{{dt}}| {_{t = {t_\varepsilon } > 0}}  = 0$.

Define ${\Phi _\varepsilon }:M \to {\mathcal{N}_\varepsilon }$ by
\[
{\Phi _\varepsilon }( y )
: = {t_\varepsilon }{\psi _{\varepsilon ,y}}.
\]
As we prove  Lemma \ref{2.11.}, we have

\begin{lemma}\label{3.9.}
Uniformly for $y \in M$, we have
\begin{equation}\label{3.21}
\mathop {\lim }\limits_{\varepsilon  \to {0^ + }} {J_\varepsilon }( {{\Phi _\varepsilon }( y )} ) = {c_{{V_0}}}.
\end{equation}
\end{lemma}

Consider $\delta  > 0$ such that ${M_\delta } \subset \Lambda $ and
choose $\rho  = \rho ( \delta  ) > 0$ satisfying ${M_\delta }
\subset {B_\rho }( 0 )$. Let $\Upsilon :{\R^3} \to {\R^3}$ be
defined as $\Upsilon ( z ): = z$ for $| z | \le \rho $ and $\Upsilon
( z ): = {\rho z/|z|}$ for $| z | \ge \rho $, and consider the map
${\beta _\varepsilon }:{\mathcal{N}_\varepsilon } \to {\R^3}$ given
by
\[
{\beta _\varepsilon }( u ): = \frac{{\int_{{\R^3}} {\Upsilon (
{\varepsilon z} ){u^2}} }} {{\int_{{\R^3}} {{u^2}} }}.
\]
Moreover, we conclude that
\begin{equation}\label{3.22}
\mathop {\lim }\limits_{\varepsilon  \to {0^ + } } {\beta _\varepsilon }
( {{\Phi _\varepsilon }( y )} ) = y{\text{ uniformly for }}y \in M.
\end{equation}
In fact, Let  $z' = \frac{{\varepsilon z - y}} {\varepsilon }$, we
see
\[{\beta _\varepsilon }( {{\Phi _\varepsilon }( y )} )
= y + \frac{{\int_{{\R^3}} {( {\Upsilon ( {\varepsilon z' + y} ) -
y} ) {\eta ^2}( {\sqrt \varepsilon  z'} ){w^2}( {z'} )} }}
{{\int_{{\R^3}} {{\eta ^2}( {\sqrt \varepsilon  z'} ){w^2}( {z'} )}
}}.\] Direct calculations show that,
\[
\int_{{\R^3}} {{\eta ^2}( {\sqrt \varepsilon  z'} ){w^2}( {z'} )}
\to \int_{{\R^3}} {{w^2}}  > 0{\text{ as }}\varepsilon  \to 0.
\]
Since $y \in M$ and $M$ is compact,
\[\begin{gathered}
  \Bigl| {\int_{{\R^3}} {( {\Upsilon ( {\varepsilon z' + y} ) - y} )
  {\eta ^2}( {\sqrt \varepsilon  z'} ){w^2}( {z'} )} } \Bigr| \hfill \\
   \le \int_{{B_{\frac{2}
{{\sqrt \varepsilon  }}}}( 0 )} {| {\Upsilon ( {\varepsilon z' + y} )
 - \Upsilon ( y )} |{w^2}( {z'} )}  \hfill \\
   \le o( 1 )\int_{{B_{\frac{2}
{{\sqrt \varepsilon  }}}}( 0 )} {{w^2}( {z'} )}
\to 0{\text{ as }}\varepsilon  \to 0{\text{ uniformly for }}y \in M. \hfill \\
\end{gathered} \]
Hence \eqref{3.22} holds.

\begin{lemma}\label{3.10.}
Let ${\varepsilon _n} \to {0^ + }$ and ${u_n} \in
{\mathcal{N}_{{\varepsilon _n}}}$ such that ${J_{{\varepsilon _n}}}(
{{u_n}} ) \to {c_{{V_0}}}$. Then there exists $\{ {{y_n}} \} \subset
{\R^3}$ such that the sequence ${u_n}( {z + {y_n}} )$ has a
convergent subsequence in ${H^1}( {{\R^3}} )$. Moreover, up to a
subsequence, ${\varepsilon _n}{y_n} \to y \in M$.
\end{lemma}

\begin{proof}
Direct calculations show that $\{ {{u_n}} \}$ is bounded in ${H^1}(
{{\R^3}} )$, the same arguments employed in Lemma \ref{2.7.}
provides a sequence $\{ {{y_n}} \} \subset {\R^3}$ and positive
constants $R$, $\beta$ such that
\[
\int_{{B_R}( {{y_n}} )} {{{| {{u_n}} |}^2}}  \ge \beta  > 0.
\]
Denote by ${{\tilde u}_n}( z ) = {u_n}( {z + {y_n}} )$,
going if necessary to a subsequence, we can assume that
\begin{equation}\label{3.23}
{{\tilde u}_n} \rightharpoonup \tilde u \ne 0{\text{ in }}{H^1}(
{{\R^3}} ).
\end{equation}
Let ${t_n} > 0$ be such that ${t_n}{{\tilde u}_n} \in
{\mathcal{N}_{{V_0}}}$, where ${\mathcal{N}_{{V_0}}}: = \{ {u \in
{H^1}( {{\R^3}} )\backslash \{ 0 \}| {\langle {{{I'}_{{V_0}}}( u
),u} \rangle  = 0} } \}$. By the definition of ${I_{{V_0}}}$,
${c_{{V_0}}}$, we obtain
\[
{c_{{V_0}}} \le {I_{{V_0}}}( {{t_n}{{\tilde u}_n}} )
= {I_{{V_0}}}( {{t_n}{u_n}} ) \le {J_{{\varepsilon _n}}}( {{t_n}{u_n}} ) \le {J_{{\varepsilon _n}}}( {{u_n}} ) = {c_{{V_0}}} + o( 1 ),
\]
from which it follows that ${I_{{V_0}}}( {{t_n}{{\tilde u}_n}} ) \to {c_{{V_0}}}$.

We claim, up to a subsequence, that ${t_n} \to {t_0} > 0$.
 Direct computations show that $\{ {{t_n}{{\tilde u}_n}} \}$
 is bounded in ${{H^1}( {{\R^3}} )}$. Since ${{\tilde u}_n}$
 does not converge to $0$ in ${H^1}( {{\R^3}} )$,
 there exists a $\delta ' > 0$
 such that ${\| {{{\tilde u}_n}} \|_{{H^1}( {{\R^3}} )}} \ge \delta ' > 0$.
 Therefore, $0 < {t_n}\delta ' \le {\| {{t_n}{{\tilde u}_n}} \|_{{H^1}( {{\R^3}} )}} \le C$.
 Thus $\{ {{t_n}} \}$ is bounded and we can
 suppose that ${t_n} \to {t_0} \ge 0$.
 If ${t_0} = 0$, in view of the boundedness of $\{ {{{\tilde u}_n}} \}$ in ${{H^1}( {{\R^3}} )}$,
 we have ${t_n}{{\tilde u}_n} \to 0$ in ${{H^1}( {{\R^3}} )}$.
 Hence ${I_0}( {{t_n}{{\tilde u}_n}} ) \to 0$,
 which contradicts ${c_{{V_0}}} > 0$.

Denote by ${{\hat u}_n}: = {t_n}{{\tilde u}_n}$,
$\hat u: = {t_0}\tilde u$, we have
\begin{equation}\label{3.24}
{I_{{V_0}}}( {{{\hat u}_n}} ) \to {c_{{V_0}}},{\text{ }}{{\hat u}_n}
\rightharpoonup \hat u{\text{ in }}{H^1}( {{\R^3}} ).
\end{equation}
In fact, by the Ekeland's Variational Principle in \cite{e},
there exists a sequence $\{ {{{\hat w}_n}} \} \subset {\mathcal{N}_{{V_0}}}$ satisfying
\begin{equation}\label{3.25}
{{\hat w}_n} - {{\hat u}_n} \to 0{\text{ in }}{H^1}( {{\R^3}} ),
{\text{ }}{I_{{V_0}}}( {{{\hat w}_n}} ) \to {c_{{V_0}}},
{\text{ }} {\| {{{I'}_{{V_0}}}( {{{\hat w}_n}} )} \|_*} \to 0.
\end{equation}

Using the same arguments as in the proof of Proposition \ref{3.5.}, we get that
\begin{equation}\label{3.26}
{{I'}_{{V_0}}}( {{{\hat w}_n}} ) \to 0{\text{ as }}n \to \infty .
\end{equation}
By \eqref{3.24}, \eqref{3.25},
\begin{equation}\label{3.27}
{{\hat w}_n} \rightharpoonup \hat u{\text{ in }}{H^1}( {{\R^3}} ).
\end{equation}
Using the same arguments as in the proof of \eqref{2.19},
we conclude that $\hat u \in {\mathcal{N}_{{V_0}}}$. Hence
\begin{eqnarray*}
  {c_{{V_0}}} &\le& {I_{{V_0}}}( {\hat u} )
  = {I_{{V_0}}}( {\hat u} ) - \frac{1}
{4}\langle {{{I'}_{{V_0}}}( {\hat u} ),\hat u} \rangle  \\
&=& \frac{1} {4}\int_{{\R^3}} {( {a{{| {\nabla \hat u} |}^2} +
{V_0}{{( {\hat u} )}^2}} )}  + \int_{{\R^3}} {\Bigl( {\frac{1} {4}f(
{\hat u} )\hat u - F( {\hat u} )} \Bigr)}  + \frac{1}
{{12}}\int_{{\R^3}} {{{( {{{\hat u}^ + }} )}^6}}   \\
 &\le& \mathop {\underline {\lim } }\limits_{n \to \infty } \frac{1}
{4}\int_{{\R^3}} {( {a{{| {\nabla {{\hat w}_n}} |}^2} + {V_0}{{(
{{{\hat w}_n}} )}^2}} )}  + \int_{{\R^3}} {\Bigl( {\frac{1} {4}f(
{{{\hat w}_n}} ){{\hat w}_n} - F( {{{\hat w}_n}} )} \Bigr)}  +
\frac{1}
{{12}}\int_{{\R^3}} {{{( {\hat w_n^ + } )}^6}}   \\
  &=& \mathop {\underline {\lim } }\limits_{n \to \infty } {I_{{V_0}}}( {{{\hat w}_n}} )
  - \frac{1}
{4}\langle {{{I'}_{{V_0}}}( {{{\hat w}_n}} ),{{\hat w}_n}} \rangle =
{c_{{V_0}}}.
\end{eqnarray*}
Thus
\[
\int_{{\R^3}} {( {a{{| {\nabla {{\hat w}_n}} |}^2} + {V_0}{{(
{{{\hat w}_n}} )}^2}} )}  \to \int_{{\R^3}} {( {a{{| {\nabla \hat u}
|}^2} + {V_0}{{( {\hat u} )}^2}} )} {\text{ as }}n \to \infty ,
\]
which  combined with \eqref{3.25} and \eqref{3.27} yields
\begin{equation}\label{3.28}
{{\tilde u}_n} \to \tilde u{\text{ in }}{H^1}( {{\R^3}} ).
\end{equation}

Now, we are going to prove that ${\varepsilon _n}{y_n} \to y \in M$.
First, as we prove Lemma~\ref{3.2.}, we can  prove that $\{
{{\varepsilon _n}{y_n}} \}$ is bounded and ${\varepsilon _n}{y_n}
\to y\in \overline {\Lambda '} $. Hence  it suffices to show that
$V( y ) = {V_0}: = \mathop {\inf }\limits_\Lambda  V$. Arguing by
contradiction again, we assume that $V( y ) > {V_0}$. Recalling
\eqref{3.28}, we get that
\begin{eqnarray*}
  {c_{{V_0}}} &=& {I_{{V_0}}}(\hat u) \\
   &<& \frac{1}
{2}a\int_{{\R^3}} {|\nabla \hat u{|^2}}  + \frac{1} {2}\int_{{\R^3}}
{V(y){{(\hat u)}^2}}  + \frac{1} {4}b{\Bigl( {\int_{{\R^3}} {|\nabla
\hat u{|^2}} } \Bigr)^2} - \int_{{\R^3}} {\Bigl( {F(\hat u) +
\frac{1}
{6}{{({{\hat u}^ + })}^6}} \Bigr)}   \\
   &\le& \mathop {\underline {\lim } }\limits_{n \to \infty } \frac{1}
{2}a\int_{{\R^3}} {|\nabla {{\hat u}_n}{|^2}}
 + \frac{1}
{2}\int_{{\R^3}} {V({\varepsilon _n}z + {\varepsilon
_n}{y_n}){{({{\hat u}_n})}^2}}
 + \frac{1}
{4}b{\Bigl( {\int_{{\R^3}} {|\nabla {{\hat u}_n}{|^2}} } \Bigr)^2}  \\
   &&- \int_{{\R^3}} {\Bigl( {F({{\hat u}_n})
   + \frac{1}
{6}{{(\hat u_n^ + )}^6}} \Bigr)}   \\
   &\le& \mathop {\underline {\lim } }\limits_{n \to \infty } {J_{{\varepsilon _n}}}({t_n}{u_n}) \le \mathop {\underline {\lim } }\limits_{n \to \infty } {J_{{\varepsilon _n}}}({u_n})
   = {c_{{V_0}}}, \\
\end{eqnarray*}
which does not make sense, thus $V(y) = {V_0}$ and the proof is
completed.
\end{proof}

Define
\[
{\widetilde{\mathcal N}_\varepsilon }
: = \{ {u \in {{\mathcal N}_\varepsilon }| {{J_\varepsilon }( u ) \le {c_{{V_0}}}
+ h( \varepsilon  )} } \},
\]
where $h( \varepsilon  ): = \mathop {\sup }\limits_{y \in M} | {{J_\varepsilon }
( {{\Phi _\varepsilon }( y )} ) - {c_{{V_0}}} } |$.
we can deduce from Lemma \ref{3.9.} that, $h( \varepsilon  ) \to 0$ as
 $\varepsilon  \to {0^ + }$. By the definition of $h( \varepsilon  )$, we know that,
  for any $y \in M$ and $\varepsilon  > 0$,
  ${\Phi _\varepsilon }( y ) \in {\widetilde{\mathcal  N}_\varepsilon }$
  and ${\widetilde{\mathcal N}_\varepsilon } \ne \emptyset $.\\

\begin{lemma}\label{3.11.}
For any $\delta  > 0$, we have
\[
\mathop {\lim }\limits_{\varepsilon  \to {0^ + }} \mathop {\sup }\limits_{u \in {{\widetilde{\mathcal N}}_\varepsilon }} {\rm{dist}}( {{\beta _\varepsilon }( u ),{M_\delta }} ) = 0.
\]
\end{lemma}

\begin{proof}
Let $\{ {{\varepsilon _n}} \} \subset R$ be such
that ${\varepsilon _n} \to {0^ + }$. By definition,
there exists ${u_n} \in {\widetilde{\mathcal N}_{{\varepsilon _n}}}$ such that
\[
{\rm{dist}}( {{\beta _{{\varepsilon _n}}}( {{u_n}} ),{M_\delta }} )
= \mathop {\sup }\limits_{u \in {{\widetilde{\mathcal N}}_{{\varepsilon _n}}}} {\rm{dist}}( {{\beta _{{\varepsilon _n}}}( u ),{M_\delta }} ) + o( 1 ).
\]
Thus it suffices to find a sequence
$\{ {{{\tilde y}_n}} \} \subset {M_\delta }$ such that
\begin{equation}\label{3.30}
| {{\beta _{{\varepsilon _n}}}( {{u_n}} ) - {{\tilde y}_n}} |
= o( 1 ).
\end{equation}
Since ${u_n} \in {\widetilde{\mathcal N}_{{\varepsilon _n}}} \subset {{\mathcal N}_{{\varepsilon _n}}}$,
we can use the definition of ${\widetilde{\mathcal N}_{{\varepsilon _n}}}$ to obtain
\[
{c_{{V_0}}}  \le \mathop {\inf }\limits_{u \in {\mathcal{N}_{{\varepsilon _n}}}} {J_{{\varepsilon _n}}}( u ) \le {J_{{\varepsilon _n}}}( {{u_n}} ) \le {c_{{V_0}}}
+ h( {{\varepsilon _n}} ),
\]
therefore, ${J_{{\varepsilon _n}}}( {{u_n}} ) \to {c_{{V_0}}} $.
 By Lemma \ref{3.10.}, we can get a sequence $\{ {{y_n}} \}$ and
 $\tilde u \in {H^1}( {{\R^3}} )\backslash \{ 0 \}$ such that
\begin{equation}\label{3.31}
{u_n}( {z + {y_n}} ) \to \tilde u{\text{ in }}{H^1}( {{\R^3}} ).
\end{equation}
Moreover, up to a subsequence,
\[
{{\tilde y}_n}: = {\varepsilon _n}{y_n} \to y \in M \subset {M_\delta }
\]
By direct computations,
\[
{\beta _{{\varepsilon _n}}}( {{u_n}} ) = {{\tilde y}_n} +
\frac{{\int_{{\R^3}} {( {\Upsilon ( {{\varepsilon _n}z + {{\tilde
y}_n}} )
 - {{\tilde y}_n}} )\tilde u_n^2( {z + {y_n}} )} }}
{{\int_{{\R^3}} {u_n^2( {z + {y_n}} )} }}.
\]
By \eqref{3.31}, we have
\[
\int_{{\R^3}} {u_n^2( {z + {y_n}} )} \to \int_{{\R^3}} {{{\tilde
u}^2}}  > 0
\]
and $\{ {u_n^2( {z + {y_n}} )} \}$ is uniformly
integrable near $\infty $, i.e. $\forall \delta ' > 0$,
$\exists R > 0$ such that
\[
{\int _{{\R^3}\backslash {B_R}(0)}}u_n^2(z + {y_n})
 < \delta '/4\rho .
\]
Thus
\[
\Bigl| {\int_{{\R^3}\backslash {B_R}( 0 )} {( {\Upsilon (
{{\varepsilon _n}z + {{\tilde y}_n}} ) - {{\tilde y}_n}} )\tilde
u_n^2( {z + {y_n}} )} } \Bigr| \le 2\rho  \cdot ( \delta '/4\rho ) =
{\delta '/2}.
 \]
Since $\{ {{{\tilde y}_n}} \} \subset {M_\delta }$ and ${M_\delta }$ is compact, then
\[\begin{gathered}
  \Bigl| {\int_{{B_R}( 0 )} {( {\Upsilon ( {{\varepsilon _n}z
   + {{\tilde y}_n}} )
   - {{\tilde y}_n}} )\tilde u_n^2( {z + {y_n}} )} } \Bigr| \hfill \\
   = \Bigl| {\int_{{B_R}( 0 )} {( {\Upsilon ( {{\varepsilon _n}z
   + {{\tilde y}_n}} ) - \Upsilon ( {{{\tilde y}_n}} )} )\tilde u_n^2( {z + {y_n}} )} } \Bigr| \hfill \\
   \le o( 1 )\int_{{B_R}( 0 )} {\tilde u_n^2( {z + {y_n}} )}  \le  o( 1 )
   < {{\delta '/2}} \hfill \\
\end{gathered} \]
for all $n$ large enough. Hence \eqref{3.30} follows, the lemma is proved.
\end{proof}

\begin{proof}[\bf Proof of  Proposition \ref{3.6.}]
Given $\delta  > 0$ such that ${M_\delta } \subset \Lambda $,
we can use Lemma \ref{3.9.}, Lemma \ref{3.11.} and \eqref{3.22}
 to obtain ${\varepsilon _\delta } > 0$ such that for any
 $\varepsilon  \in ( {0,{\varepsilon _\delta }} )$, the diagram
\[
M\mathop  \to \limits^{{\Phi _\varepsilon }} {\widetilde{\mathcal N}_\varepsilon }\mathop  \to \limits^{{\beta _\varepsilon }} {M_\delta }
\]
is well defined. In view of \eqref{3.22}, for $\varepsilon$
small enough, we can denote by
${\beta _\varepsilon }( {{\Phi _\varepsilon }( y )} ) = y + \theta ( y )$ for $y \in M$,
 where
 $| {\theta ( y )} | < {\delta '/2}$ uniformly for $y \in M$.
 Define $S( {t,y} ) = y + ( {1 - t} )\theta ( y )$. Thus $S:[ {0,1} ] \times M \to {M_\delta }$
 is continuous. Obviously,
  $S( {0,y} ) = {\beta _\varepsilon }( {{\Phi _\varepsilon }( y )} )$
  and
  $S( {1,y} ) = y$ for all $y \in M$.
  That is, ${\beta _\varepsilon } \circ {\Phi _\varepsilon }$
  is homotopically equivalent to $Id:M \to {M_\delta }$. By Lemma \ref{3.8.}, we obtain that
\[
 {\rm{ca}}{{\rm{t}}_{{{\widetilde{\mathcal N}}_\varepsilon }}}
 ( {{{\widetilde{\mathcal N}}_\varepsilon }} ) \ge {\rm{ca}}{{\rm{t}}_{{M_\delta }}}( M ).
\]
Since
 ${c_{{V_0}}} < \frac{1}{4}ab{S^3} + \frac{1}{{24}}{b^3}{S^6} + \frac{1}{{24}}{( {{b^2}{S^4} + 4aS} )^{\frac{3}{2}}}$,
 we can use the definition of
 ${{{\widetilde{\mathcal N}}_\varepsilon }}$ and
 Proposition \ref{3.5.} to conclude that ${J_\varepsilon }$
 satisfies the $(P.S.)$ condition in ${{{\widetilde{\mathcal N}}_\varepsilon }}$
 for all small $\varepsilon > 0$. Therefore, Lemma \ref{3.7.}
 proves at least
 ${\rm{ca}}{{\rm{t}}_{{{\widetilde{\mathcal N}}_\varepsilon }}}
 ( {{{\widetilde{\mathcal N}}_\varepsilon }} )$
  critical points of ${J_\varepsilon }$ restricted to
  ${{{\widetilde{\mathcal N}}_\varepsilon }}$.
  Using the same arguments as in the proof of
  Proposition \ref{3.5.}, we can conclude that a critical
  point of the functional
  ${J_\varepsilon }$ on ${{\mathcal N}_\varepsilon }$,
  is in fact, a critical point of the functional ${J_\varepsilon }$
  in ${H_\varepsilon }$ and therefore a weak solution for the
  problem $({{{\hat E'}_\varepsilon }})$, the theorem is proved.
\end{proof}

\begin{proof}[\bf Proof of  Theorem \ref{1.2.}]
For any sequence $\{ {{\varepsilon _n}} \} \subset R$
satisfying ${\varepsilon _n} \to {0^ + }$, denote
${v_{{\varepsilon _n}}} \in {\widetilde{\mathcal N}_{{\varepsilon _n}}} \subset {{\mathcal N}_{{\varepsilon _n}}}$ by a weak solution
of $( {{{\hat E'}_{{\varepsilon _n}}}} )$, we can use the
definition of ${{{\widetilde{\mathcal N}}_\varepsilon }}$
to obtain
\[
{c_{{V_0}}} \le \mathop {\inf }\limits_{u \in {\mathcal{N}_{{\varepsilon _n}}}}
{J_{{\varepsilon _n}}}( u ) \le {J_{{\varepsilon _n}}}( {{v_{{\varepsilon _n}}}} )
\le {c_{{V_0}}} + h( {{\varepsilon _n}} ),
\]
therefore, ${J_{{\varepsilon _n}}}( {{v_{{\varepsilon _n}}}} ) \to
{c_{{V_0}}} $. By Lemma \ref{3.10.}, we can get a sequence $\{
{{y_n}} \} \subset {\R^3}$ and $\tilde v \in {H^1}( {{\R^3}}
)\backslash \{ 0 \}$ such that
\begin{equation}\label{3.32}
{v_{{\varepsilon _n}}}( {z + {y_n}} ) \to \tilde v{\text{ in
}}{H^1}( {{\R^3}} ).
\end{equation}
Moreover, up to a subsequence,
\[
{\varepsilon _n}{y_n} \to y \in M.
\]
\eqref{3.32} and the Sobolev's Theorem show that
\[
{v_{{\varepsilon _n}}}( {z + {y_n}} ) \to \tilde v{\text{ in
}}{L^6}( {{\R^3}} ).
\]
Thus, $\{ {{{| {{v_{{\varepsilon _n}}}( {z + {y_n}} )} |}^6}} \}$
is uniformly integrable near $\infty $. By Lemma \ref{2.11.}, we get that
\[
\mathop {\lim }\limits_{|z| \to \infty } {v_{{\varepsilon _n}}}
( {z + {y_n}} ) = 0{\text{ uniformly for }}n.
\]
Proceeding as we prove Theorem~\ref{1.1.}, we can complete the
proof.
\end{proof}

\end{document}